\documentclass[11pt]{article}
\usepackage[top=2cm,bottom=2.5cm,right=2.5cm,left=2.5cm]{geometry}
\usepackage[english]{babel}
\usepackage[T1]{fontenc}
\usepackage{times}
\usepackage{mathtools, bm}
\usepackage{amssymb, bbm}
\usepackage{graphicx}
\usepackage{mathrsfs}
\usepackage{stmaryrd}
\usepackage{amsthm}
\usepackage{listings}
\usepackage{pict2e}
\usepackage{pgf, tikz}
\usepackage{dsfont}
\usepackage{algorithm2e}
\usepackage{algorithmic}
\usepackage{multicol}
\usepackage{hyperref}
\usepackage{subcaption}

\newcommand{\E}{\mathbb E}
\newcommand{\R}{\mathbb R} 
 
\newcommand{\C}{\mathcal C}
 
\newcommand{\A}{\mathcal A} 
\newcommand{\n}{\mathbb N} 
\newcommand{\N}{\mathbb N}

\newcommand{\dd}{\text{d}}

\newcommand{\X}{\mathbf{X}}

\newcommand{\po}{\left(}
\newcommand{\pf}{\right)}

\DeclareMathOperator \ch {ch} 
\DeclareMathOperator \sh {sh} 

\allowdisplaybreaks

\thicklines

{
      \theoremstyle{plain}
      
  }

\newtheorem{theorem}{Theorem}
\newtheorem{proposition}{Proposition}

\newtheorem{lemma}{Lemma}
\newtheorem{remark}{Remark}

\numberwithin{equation}{section} 
\numberwithin{lemma}{section} 
\numberwithin{remark}{section} 
\numberwithin{example}{section}
\numberwithin{proposition}{section}

\date{}
\author{Lucas Journel\footnote{Institut de Mathématiques, Université de Neuchâtel, 11 Rue Emile-Argand, 2000 Neuchâtel, Suisse.\linebreak \textit{Email: }lucas.journel[AT]unine.ch},$\quad $ Pierre Le Bris\footnote{Institut des Hautes Etudes Scientifiques, 35 route de Chartres, CS 40001  91893 Bures-sur-Yvette, France.\linebreak \textit{Email: }lebris[AT]ihes.fr}}
\title{On uniform in time propagation of chaos in metastable cases: the Curie-Weiss model}

\begin{document}

\maketitle

\begin{abstract}
Many low temperature particle systems in mean-field interaction are ergodic with respect to a unique invariant measure, while their (non-linear) mean-field limit may possess several steady states. In particular, in such cases, propagation of chaos (i.e. the convergence of the particle system to its mean-field limit as $n$, the number of particles, goes to infinity) cannot hold uniformly in time since the long-time behaviors of the two processes are a priori incompatible.

However, the particle system may be metastable, and the time needed to exit the basin of attraction of one of the steady states of its limit, and go to another, is exponentially (in $n$) long. Before this exit time, the particle system reaches a (quasi-)stationary distribution, which we expect to be a good approximation of the corresponding non-linear steady state.

Our goal is to study the typical metastable behavior of the empirical measure of such mean-field systems, starting in this work with the Curie-Weiss model. We thus show uniform in time propagation of chaos of the spin system conditioned to keeping a positive magnetization.
\end{abstract}

%
%
%
%

\section{Introduction}

\subsection{About metastability and propagation of chaos}

Consider a system of $n$ particles in mean-field interaction. In some low temperature cases, on which we  say more later, the  system is ergodic with respect to a unique invariant measure, while its (non-linear) mean-field limit possesses several steady states. Even though this should a priori prevent the convergence of the particle system to its limit (i.e. as $n\rightarrow\infty$) uniformly in time, in those cases, the particle system is metastable, and the time needed to exit a  basin of attraction of one of the steady states of its limit, and go to another, is exponentially (in $n$) long. Before this exit time, the particle system reaches a (quasi-)stationary distribution, which we expect to still be a good approximation of the corresponding non-linear steady state. The goal of this article is to give a proof of concept for this idea on a toy model.

For a given mean-field particle system $\X_t = (X_t^1,\dots,X_t^n)$, we define its empirical measure as 
\begin{equation}
    \pi(\X_t) = \frac{1}{n}\sum_{i=1}^n \delta_{X^i_t}.
\end{equation}
This is a probability measure that counts the number of particles in any given set. As $n$ goes to infinity, we expect the empirical measure to converge to the solution of a non-linear PDE of the form
\begin{equation}\label{eq:non-linear-limit-gen}
\partial_t\rho = L_\rho^*\rho,
\end{equation}
provided $\pi(\X_0)$ converges to $\rho_0$ and the particle system is exchangeable. This is the propagation of chaos phenomenon, which has been extensively studied (see \cite{Kac56, McK66, Szn91, Mel96} for some historical milestones), and we refer to \cite{CD22-1,CD22-2} for a recent in-depth review. As we have said, a sine qua non for uniform in time propagation of chaos is the stability of the non-linear limit, in the sense that there must be at most one solution to
\begin{equation}\label{eq:steady-states-gen}
L_\rho^*\rho =0    
\end{equation}
and, if there is a solution to \eqref{eq:steady-states-gen}, that any solution of~\eqref{eq:non-linear-limit-gen} should converge to this steady state. In generic cases, this would imply that the particles system converges to its own invariant distribution at a rate which is independent of the number of particles. 

In this work, we are interested in the case where~\eqref{eq:steady-states-gen} admits several solutions, and some of them are locally stable. In this case, the propagation of chaos could still hold, but would not be uniform in time. The idea is that, while the solution of the non-linear limit would not leave the basin of attraction it started in, the empirical measure of the particle system could still go from one basin to another, though in a time that is exponentially long in $n$. This can be shown via the Eyring-Kramers formula, see for instance \cite[Chap. 13]{BdH15} for the Curie-Weiss model. Before exiting such a basin, the law of the empirical measure can converge towards a quasi-stationary distribution (QSD). This is known as metastability, and is a well-documented phenomenon for many stochastic processes (see for instance~\cite{LMS21,AKT23,BG16,DgLLpN20,DgLLpN22}). To be more precise, denoting $\mathcal{P}$ the set of probability measures on the space we consider, let us define some basin $B_{\nu}$,
\[
B_{\nu} = \left\{ \rho_0 \in \mathcal P, \, \lim_{t\rightarrow\infty}\rho = \nu\right\}
\]
where $\rho$ is the solution to~\eqref{eq:non-linear-limit-gen} with initial condition $\rho_0$, for a given solution $\nu\in \mathcal{P}$ to~\eqref{eq:steady-states-gen}, and write, for some large enough sub-domain $D\subset B_{\nu}$
\[
\tau_n = \inf\left\{t\geqslant 0, \pi(\X_t) \notin D \right\}.
\]
Then the QSD, defined as the limit
\begin{equation}\label{eq:def-QSD-intro}
\nu^n_\infty = \lim_{t\rightarrow\infty} \mathcal Law(\pi(\X_t)|\tau_n>t),
\end{equation}
corresponds to the law of $\pi(\X_t)$ for large time scales, however smaller than the typical exit time. In particular, it is expected that 
\[
\E_{\mu\sim\nu^n_\infty}\po \mu(f) \pf \simeq \nu(f)
\] 
in the limit $n\rightarrow\infty$, as the typical exit time goes to infinity.
This article, which studies the case of the Curie-Weiss model, is the first in a line of works which aims at showing that the evolution of such particle systems inside a basin of attraction can be described uniformly in time by the non-linear limit, as shown by a convergence of the form
\begin{align*}
\sup_{t\geqslant 0}\left|\E\left(\frac{1}{n}\sum_{i=1}^nf(X_t^i)\middle|\tau_n>t\right)-\rho_t(f)\right|\xrightarrow[n\rightarrow \infty]{} 0.
\end{align*}
Note that the domain of attraction of an equilibrium is usually unknown, so that we cannot take $D=B_{\nu}$. Even if $B_{\nu}$ is known explicitly (as it is the case for the Curie-Weiss model), for technical reasons, we shall consider sub-domains. However, for our result to hold, $D$ must be also a metastable set for the dynamic of the empirical measure.

%
%
%
%

\subsection{The Curie-Weiss model and main results}

Let us now introduce the particular model that we are interested in, the Curie-Weiss model. Fix $n\in\n$ and consider a set of $n$ individuals. To each individual is associated a spin (or opinion), denoted  ${X^{i,n}_t\in\{-1,1\}}$, that evolves with time $t\geqslant 0$. Let $\textbf{X}_t=(X^{1,n}_t,...,X^{n,n}_t)$ and let $\beta>0$ be a parameter, called the inverse temperature, that measures the tendency of each individual to change spin. Finally, denote $\mathcal{S}_n=\{-1,1\}^n$ the set of spin configurations. We consider a \textit{spin-flip dynamics} on $\mathcal{S}_n$, commonly known as Glauber dynamics, according to which, any time $t$, the system may switch its i-th spin at rate 
\[
\exp\left(-\frac{\beta}{n} X^{i,n}_t\sum_{j=1}^nX^{j,n}_t\right).
\]
This creates a dynamics where the spins try and align themselves with the others.
Denote the empirical measure
\begin{equation*}
    \pi^n_t=\frac{1}{n}\sum_{i=1}^n\delta_{X^{i,n}_t},
\end{equation*}
which, because of the binary nature of the state space, can be identified with the empirical mean, also known as the \textit{magnetization} of the system, 
\begin{equation*}
    m_t^n:=\pi_t^n(x\mapsto x) = \frac{1}{n}\sum_{i=1}^nX^{i,n}_t.
\end{equation*}
From the spin-flip dynamics, we obtain that this quantity is a Markov chain on the finite state space
\begin{align*}
    E_n=\left\{x\ :\ \exists i\in\{0,...,n\}, x=-1+\frac{2i}{n}\right\},
\end{align*}
with transition rates for the empirical measure, or equivalently for the magnetization, given by
\begin{align*}
    \lambda_+^n(m)=n\frac{1-m}{2}e^{\beta m}\ \ \ \text{ and }\lambda_-^n(m)=n\frac{1+m}{2}e^{-\beta m},
\end{align*}
so that the generator of the process $m_t^n$ is
\begin{equation}\label{eq:gen_mag}
    \mathcal{A}_nf(m)=\lambda_+^n(m)\left(f\left(m+\frac{2}{n}\right)-f\left(m\right)\right)+\lambda_-^n(m)\left(f\left(m-\frac{2}{n}\right)-f\left(m\right)\right).
\end{equation}
Formally, we can check that, as $n\rightarrow\infty$, the generator $\mathcal A_n$ converges towards
\[
\A f(m) = \left((1-m)e^{\beta m}-(1+m)e^{-\beta m}  \right)f'(m) = -g'(m)f'(m),
\]
where
\begin{equation}\label{eq:pot-CW}
g(m) = C - \frac{1}{\beta}\left(1+\frac{1}{\beta} + m\right)e^{-\beta m} - \frac{1}{\beta}\left(1+\frac{1}{\beta} - m\right)e^{\beta m},
\end{equation}
for $-1\leqslant m\leqslant 1$, and the constant $C\geqslant0$ above is chosen such that $\min g=0$. In particular, the propagation of chaos for the Curie-Weiss model reads as the convergence of the magnetization $m^n$ to the solution of 
\begin{equation}\label{eq:non-linear-limit}
    \frac{d}{dt}\overline m_t=-g'(\overline{m}_t)=(1-\overline m_t)e^{\beta \overline m_t}-(1+\overline m_t)e^{-\beta \overline m_t},
\end{equation}
see Proposition~\ref{prop:PoCnonUnif} below. The solution of this ordinary differential equation (ODE) is denoted $\overline m_t$ in all of this work. The drift $g$ acts as a confining potential and one can check that, if $\beta\leqslant 1$, $g$ admits a unique minimizer at $m=0$, and if $\beta>1$, $g$ admits a minimum reached in $m_+$ and $m_-$, which satisfy $m_-=-m_+$, as well as one unique additional critical point at $m=0$, see Figure~\ref{fig:potentiel_g} or~\cite{BdH15}. Hence, at low temperature ($\beta>1$), for all $\overline m_0>0$, $\overline m_t$ converges to $m_+$. However, $m^n_t$, for all $t>0$, has a positive density on the entire state space $E_n$, which prevents the uniformity in time in the convergence of $m^n$ to $\overline m_t$, as made precise again in Proposition~\ref{prop:PoCnonUnif}. 


\begin{figure}
	\begin{subfigure}{.5\textwidth}
    	\centering
	\includegraphics[width=\linewidth,,height=0.75\linewidth]{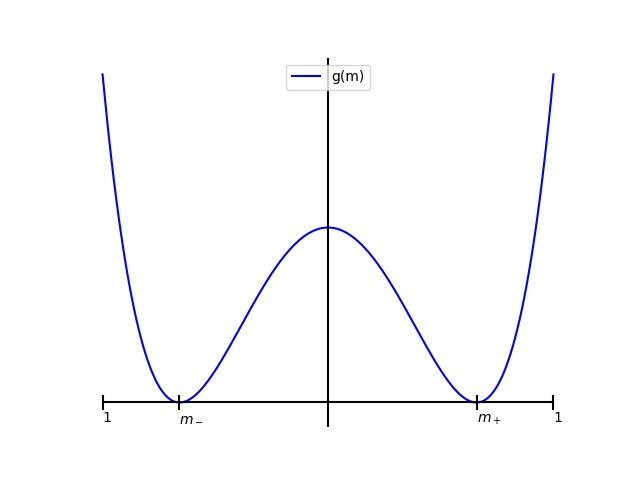}
	 \end{subfigure}
	 \begin{subfigure}{.5\textwidth}
    	\centering
   	\includegraphics[width=\linewidth,,height=0.75\linewidth]{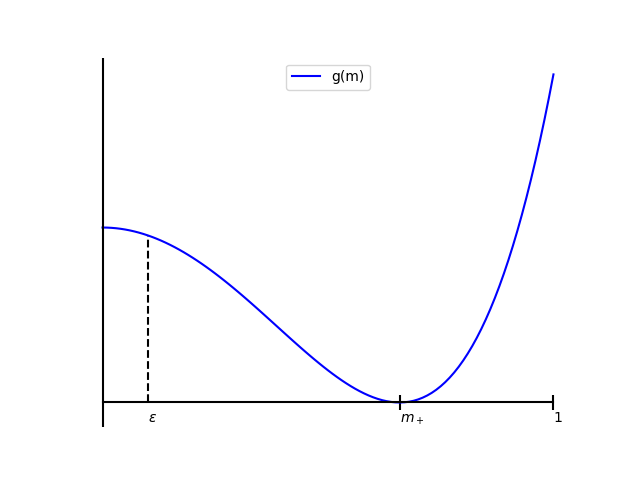}
	 \end{subfigure}
	\caption{\textbf{Left:} Potential $g$ defined in \eqref{eq:pot-CW} for $\beta=1.2$. There are two minimizers denoted $m_-$ and $m_+$. \textbf{Right:} The process is conditioned to staying in $[\varepsilon,1]$.}
 \label{fig:potentiel_g}
\end{figure}

In order to state our results, let us now fix a few notations. For a random variable $m$, we define $\mathcal{L}(m)$ the probability distribution of $m$. Additionally, for an event $A$, we define  $\mathcal{L}(m\big|A)$ the law of $m$ conditionally on the event $A$. We similarly define $\mathbb{E}\left(\cdot\big| A\right)$ the conditional expectation. For a random process $m:[0,\infty[\mapsto[-1,1]$, we define $\mathbb{E}_{\nu}$ the expectation with respect to the process $m$ with initial distribution $m_0\sim\nu$. We similarly define $\mathbb{P}_{\nu}$. Whenever $\nu$ is a Dirac mass at a point $x$, we may, with a slight abuse of notations, write $\mathbb{E}_{x}=\mathbb{E}_{\delta_x}$.

Let us also define (and give some classical results on) the various distances between probability measures that we use throughout the article. For $\mu,\nu\in\mathcal{P}(\mathbb{R})$, denote $\Pi\left(\mu,\nu\right)$ the set of probability measures $\pi\in\mathcal{P}(\mathbb{R}\times\mathbb{R})$ with marginal distributions $\mu$ and $\nu$. We define the \textit{Wasserstein} (1 and 2) and the \textit{Total Variation} distances by
\begin{align}
\mathcal{W}_1\left(\mu,\nu\right):=&\inf_{\pi\in \Pi\left(\mu,\nu\right)}\int_\mathbb{R}|x-y|d\pi(x,y)=\inf_{(X,Y)\sim \pi,\ \pi\in \Pi\left(\mu,\nu\right)}\mathbb{E}|X-Y|\nonumber\\
=&\sup_{f : \|f\|_{\textrm{lip}}\leqslant 1}\int f(x)(d\mu-d\nu)(x),\nonumber\\
\mathcal{W}_2\left(\mu,\nu\right):=&\left(\inf_{\pi\in \Pi\left(\mu,\nu\right)}\int_\mathbb{R}|x-y|^2d\pi(x,y)\right)^{1/2}=\inf_{(X,Y)\sim \pi,\ \pi\in \Pi\left(\mu,\nu\right)}\mathbb{E}\left(|X-Y|^2\right)^{1/2},\nonumber\\
d_{TV}(\mu,\nu):=&\sup_{f : \|f\|_{\infty}\leqslant 1}\int f(x)(d\mu-d\nu)(x).\label{eq:def_distances}
\end{align}
The equivalence between the given definitions for each distance is classical and can be found for instance in \cite{Vil09}.
We may now state the first convergence results.


\begin{proposition}[Finite time PoC] \label{prop:PoCnonUnif}
Denote $\mathcal{L}(m_t^n)$ (resp. $\mathcal{L}(\overline{m}_t)$) the law of the process $m_t^n$ (resp. $\overline{m}_t$ for which, even though it is deterministic, we allow for a random initial condition $\overline{m}_0$). There exist $c,C>0$ such that for all $t\geqslant 0$ and $n\in\mathbb{N}$
\begin{equation}\label{eq:PoC_pas_unif}
\mathcal{W}_2(\mathcal{L}(m_t^n),\mathcal{L}(\overline{m}_t))\leqslant e^{ct}\left(\mathcal{W}_2(\mathcal{L}(m_0^n),\mathcal{L}(\overline{m}_0))+\frac{C}{\sqrt{n}}\right).
\end{equation}
Furthermore, if $\beta<1$, there exist $c,C>0$ such that for all $t\geqslant0$ and for all $n\in\mathbb{N}$
\begin{align*}
\mathcal{W}_2(\mathcal{L}(m_t^n),\mathcal{L}(\overline{m}_t))\leqslant e^{-ct}\mathcal{W}_2(\mathcal{L}(m_0^n),\mathcal{L}(\overline{m}_0))+\frac{C}{\sqrt{n}},
\end{align*}
and if $\beta>1$ we have
\begin{equation}\label{eq:contrexemple_PoC_unif}
 \exists \eta>0,\ \forall n\in\mathbb{N},\ \forall m_0^n=\overline{m}_0>0,\quad \sup_{t\geqslant0}\mathcal{W}_2(\mathcal{L}(m_t^n),\delta_{\overline{m}_t})\geqslant \eta
\end{equation}
\end{proposition}
This proposition, which is proved in Section~\ref{s-sec:qques_resultats_CW}, tells us that, for any compact time interval, the process $m^n$ can be well approximated by $\overline m$, or conversely. However, on large time scales and for $\beta>1$, they would exhibit very different behaviors. The time scale is given by Eiring-Kramers formula, \cite[Chap. 13]{BdH15}. Indeed, for $\beta>1$, the potential $g$ exhibits two wells (around $m_+$ and $m_-$, see Figure~\ref{fig:potentiel_g}) and the limit $\overline{m}$ with initial condition $\overline{m}_0>0$ is bound to stay in the positive well. The magnetization $m^n$ will however almost surely leave the positive well to explore the negative one, as its unique stationary distribution, to which it converges, is symmetric. Our goal is now to describe the fact that, however long it takes for the process $m^n$ starting from a positive initial condition to go to $[-1,0]$, it will still be close to $\overline{m}$ as long as it remains in $[0,1]$. To this end, let us once and for all fix $0<\varepsilon<m_+$ and write the stopping time
\begin{align*}
\tau_n=\inf\{t\geqslant 0\text{ s.t. }m_t^n\leqslant \varepsilon\}.
\end{align*}
For the sake of simplicity, we assume $\varepsilon\in \R\setminus\mathbb{Q}$ to ensure that for all $n\in\n$, $\varepsilon\notin E_n$. Throughout the article, we call $\tau_n$ a \textit{death time}, as we force the process $m^n$ to stay greater than $\varepsilon$. Conditionally on the event $\left\{\tau_n >t\right\}$, the state space becomes
\[
E_n^{\varepsilon} = E_n \cap [\varepsilon,1].
\]
We then consider the law Curie-Weiss model, conditionally on its \textit{survival},
\begin{equation*}
\nu^n_t=\mathcal L(m_t^n|\tau_n>t).
\end{equation*}
as it is \textit{killed} when it leaves $E_n^{\varepsilon}$.  Our first result concerns the long-time convergence of $\nu^n$ towards a quasi-stationary distribution, with an explicit (in $n$) rate of convergence.


\begin{theorem}[Long-time convergence to QSD]\label{thm:long-time-behavior} 
If $\beta>1$, for all $n\in \mathbb{N}$, there exist a unique distribution $\nu^n_{\infty}$ on $E_n^{\varepsilon}$, called quasi-stationary distribution (QSD), such that the following holds:

Let $\eta\in]\varepsilon,m_+[$. There exist $c, C>0$ such that, for any $n\in\mathbb{N}$ and any initial condition $m_0^n\in[\eta,1]\cap E^\varepsilon_n$, we have for all $t\geqslant 0$
\begin{equation}\label{eq:cv_qsd}
d_{TV}(\nu^n_t,\nu^n_{\infty})\leqslant C n e^{-c t}.
\end{equation}
\end{theorem}

Notice that the QSD $\nu^n_\infty$ does not exactly correspond to the previous definition~\eqref{eq:def-QSD-intro}, but both objects are related as, in our case, $\pi^n_t$ is a function of $m^n_t$. From this estimate, combined with non uniform in time propagation of chaos results (see Section~\ref{sec:PoC}), we finally obtain the desired result.


\begin{theorem}\label{thm:final_result}
Assume $\beta>1$ and let $\eta\in]\varepsilon, m_+[.$ There exist $C, \alpha>0$ such that for any Lipschitz continuous function $f:[0,1]\mapsto\R$, and any $m_0^n\in E^\varepsilon_n\cap[\eta,1]$
\begin{align*}
\sup_{t\geqslant 0}\left|\E_{m_0^n}\left(f(m_t^n)|\tau_n>t\right)-f(\overline{m}_t)\right|\leqslant C\frac{\|f\|_{\infty}+\|f\|_{\textrm{lip}}}{n^\alpha},
\end{align*}
where $\overline{m}$ is the solution to~\eqref{eq:non-linear-limit} with initial condition $m_0^n$.
\end{theorem}

Let us give a few remarks on those results.


\begin{remark}

\begin{itemize}
    \item It can be notice that since the state space $E_n^\varepsilon$ is finite, the existence of a unique QSD, as well as the convergence of the conditional law towards this QSD, are immediate from the Perron-Frobenius Theorem, see Lemma~\ref{lem:existence-hn} for details. This however does not yield any explicit rate of convergence, and hence we could not deduce Theorem~\ref{thm:final_result} from it.

    \item Note also that, in Theorem~\ref{thm:long-time-behavior}, we start by fixing a parameter $\eta\in]\varepsilon,m_+[$ and then consider initial conditions $m_0^n\in[\eta,1]\cap E^\varepsilon_n$. The parameter $\eta$ serves to give an additional margin away from the killing point, and allows for some uniformity on the initial condition. This simplifies several calculations, mainly because we need to give controls on the death probability $\mathbb{P}\left(\tau_n>t\right)$ (where, recall, $\tau_n$ is the first hitting time of $[0,\varepsilon]$) and thus starting from the left-most point would create some difficulties.
\end{itemize}
\end{remark}

In the literature, the names \textit{Curie-Weiss} model or \textit{Glauber dynamics} seem to encapsulate a number of different models of spin dynamics, which however do share a similar phase transition and metastable behavior. We choose to consider a continuous time jump process similar to that of \cite{CDP12, Kra16, CK17}, but the results and methods could be adapted to other type of magnetization dynamics (in discrete times for instance). The main feature of interest to us is that it corresponds to a well known metastable process, which can be described by a finite dimensional secondary process, the magnetization.

The Curie-Weiss model, or equivalent models, and the large deviations from its limit have been widely studied  (see for instance \cite{EN78, CK17}, \cite[Chap. 4]{Ell85} or \cite[Chap. 2]{FV17}), and in particular the metastability is well-known \cite{CGOV84}. For other works on the metastability of the Glauber dynamics in related models, let us mention for instance \cite{BEGK01} (which estimates transition times for a class of random walks in multi-well potentials are obtained), \cite{BEGK02} (in a more abstract setting, but still in a discrete state space), \cite{BM02} (for the related Ising model in dimensions 2 and 3), \cite{BMP21} (study of the exit time for the dilute Curie–Weiss model) or \cite{Bas19}.

Concerning the long-time convergence of the system, the work \cite{LLP10} shows that the mixing time -for a form of Glauber dynamics restricted to positive magnetizations- is of order $n\log n$. Note that this is not in contradiction with Theorem~\ref{thm:long-time-behavior}, as we study the convergence of the magnetization whereas \cite{LLP10} considers the law of the process on the entire state space $\mathcal{S}_n$. The idea of studying the dependency in $n$ of the rate of the long time convergence of a particle system in order to show uniform in time propagation of chaos may also be found in~\cite{Mon23,JM22,JM25} for instance.

Finally, while uniform in time propagation of chaos (or mean-field limit) has attracted a lot of attention in the recent years (see for instance \cite{DEGZ20, GM21, GLBM22, LLF23, CLRW24, Sch24} in the case of regular interactions), to the best of our knowledge, this is the first time that such a result is obtained through the use of the quasi-stationary distribution for a conditioned particle system. The explicit and quantitative distance of the magnetization to the non-linear process is therefore our main new contribution.

%
%
%
%

\subsection{Outline}\label{s-sec:methode}

Let us now describes the method of proof for both Theorems~\ref{thm:long-time-behavior}~and~\ref{thm:final_result}. Denote $M^n$ the non conservative semi-group associated to the killed process $(m^n_t)_{t\geqslant0}$, defined as
\begin{align*}
    M^n_tf(m) = \mathbb E_m\po f(m^n_t)\mathds{1}_{\tau_n>t}\pf.
\end{align*}
The proof of Theorem~\ref{thm:long-time-behavior} relies on a so-called Doob transform of $M^n$, already used in the study of long-time behavior of non-conservative semi-group as in~\cite{FRS21} or \cite{BCGM22}. Using Perron-Frobenius Theorem, we show, in Lemma~\ref{lem:existence-hn}, the existence of some eigen-elements $(b_n,h_n)$ such that
\[
M^n_th_n=e^{-b_nt}h_n,\qquad h_n >0.
\]
These elements allow us to define the Doob transform
\[
P^n_tf(m)=e^{b_nt}h_n^{-1}M^n_t(h_nf)(m).
\]
There, $P^n$ is a Markov semi-group, to which we apply the classical Meyn-Tweedie approach: we show a local Doeblin condition around the stable state $m_+$ (see Lemmas~\ref{lem:densite-CW}~and~\ref{lem:densite-CW-killed}), as well as the existence of a Lyapunov function (see Lemma~\ref{lem:lyapunov-CW-killed}) which tends to bring the process close to $m_+$. These results heavily  rely on the metastability of the process, which translates at the level of the eigen-elements into $b_n\rightarrow 0$ and $h_n \rightarrow 1$. Indeed, as shown in Lemma~\ref{lem:estime-bn}, $b_n$ represents the rate at which the process reaches $\varepsilon$, and in the limit $n\rightarrow\infty$, we obtain that $M^n$ becomes a Markov semi-group and that $h_n$ converges towards the associated eigenfunction, i.e. the constant $1$ under our normalisation.

Concerning the long-time convergence to the QSD, the fact that we do not obtain a uniform in $n$ convergence result is a consequence of the method we apply. 
Note that the limit $(\overline{m}_t)_t$ is deterministic. As a consequence, the set on which we create density (i.e. on which we prove the Doeblin condition) becomes smaller and smaller as $n$ grows (in fact, it is of the form $K_n=\left[m_+-n^{-1/2},m_++n^{-1/2}\right]$). This comes from the fact that a central limit theorem would yield that $m^n_t\simeq \overline{m}_t + Gn^{-1/2}$, where $G$ is some Gaussian random variable. Our Lyapunov function on the contrary does not scale with $n$ (we morally use $g$), meaning that the time necessary for it to bring the process  to the compact set $K_n$ grows with $n$, hence the non uniformity in $n$ of the long-time convergence result.
 
Thanks to Theorem~\ref{thm:long-time-behavior}, we show that the QSD converges (with an explicit rate of convergence) towards $\delta_{m_+}$. Then the proof of Theorem~\ref{thm:final_result} is as follows: for $t\gtrsim n^\alpha$ ($\alpha>0$), $\nu_t^n$ and $\delta_{m_t}$ are close to $\nu_\infty^n$ and $\delta_{m_+}$ respectively, which are themselves close to each other, and for small $t\lesssim \ln(n)$ one can use Proposition~\ref{prop:PoCnonUnif} to control the difference between $m^n$ and $\overline{m}$. It remains to show that $m^n_t$ is close to $\overline{m}_t$ for $\ln(n)\lesssim t\lesssim n^\alpha$. To this end, we show an additional control, estimate~\eqref{eq:control_int} below.

\paragraph{Future directions.}

The Curie-Weiss model is one specific model. We hope nonetheless that this work may serve as well as a proof of concept for the study of many other metastable mean-field models. One such system that we hope would exhibit similar behaviors is given by the SDE
\[
    dX^i_t = -\nabla U\po X^i_t \pf dt - \frac{1}{N}\sum_{j=1}^N \nabla W\po X^i_t - X^j_t \pf dt + \sqrt{2\sigma}dB_t, \qquad i \in \llbracket 1,N\rrbracket,
\]
where $W(x) = x^2$ and $U$ is a double-well potential (e.g. $U(x) = (1-x^2)^2$). In this case, the limit $N\rightarrow\infty$ is described by the McKean-Vlasov process
\begin{align*} 
\left\{ \begin{array}{l}
    d\overline{X}_t = -\nabla U\po \overline{X}_t \pf dt - \nabla W\ast\overline{\rho}_t\po \overline{X}_t \pf dt + \sqrt{2\sigma}dB_t, \\ 
    \overline{\rho}_t = \mathcal{L}(\overline{X}_t).
    \end{array}
\right.
\end{align*} 
For low $\sigma$, for each well of $U$, there exists a locally stable equilibrium of this non-linear process, as proved recently in~\cite{MR24,Z25} (see also \cite{Bas20}). The next question would then be to show that if the empirical measure of the particle systems is close enough to one of these equilibria, then we recover a uniform in time convergence for the killed process. Note that the transition time for such a model has been studied in \cite{BM21}.

\paragraph{Organization of the paper.}
In Section~\ref{sec:inter-result} we start by proving various results on the classical Curie-Weiss model, among which we prove Proposition~\ref{prop:PoCnonUnif}, before studying the eigen-elements of the non conservative semi-group $M^n$. Section~\ref{sec:cv_vers_qsd} is then dedicated to the proof of Theorem~\ref{thm:long-time-behavior}.
Finally, in Section~\ref{sec:PoC}, we obtain two results of propagation of chaos, namely a non uniform in time propagation of chaos result for the conditioned process~\eqref{eq:non-unif-poc} and a generation-of-chaos-like result~\eqref{eq:control_int}, from which we deduce Theorem~\ref{thm:final_result}.

%
%
%
%

\paragraph{Notations.}

We gather here a list of some notations that the reader can refer to if need be. 
\begin{itemize}
    \item $n\in\n$ is the number of spins.
    \item $\beta\geqslant 0$ is the inverse temperature of the system. In most part of the article we assume $\beta>1$.
    \item $m^n_t$ is the magnetization of a system of size $n$ at time $t\geqslant0$. It is a jump process on $E_n=\left\{-1,-1+\frac{2}{n},...,1-\frac{2}{n},1\right\}$ defined by its generator $\mathcal{A}_n$ and its transition rates $\lambda^n_+, \lambda^n_-:E_n\mapsto \mathbb{R}^+$ \eqref{eq:gen_mag}.
    \item $\overline{m}_t$ is the limit process, defined by the ODE \eqref{eq:non-linear-limit} and its generator $\overline{\mathcal{A}}$.
    \item $g:[-1,1]\mapsto\mathbb{R}^+$ is the underlying potential of the processes defined in \eqref{eq:pot-CW}. $m_+=-m_-\in[0,1]$ are then the points at which $g$ reaches its minimum $0$ for $\beta>1$.
    \item $\varepsilon\in]0,m_+[$ is a small margin, defining a subdomain $E^\varepsilon_n=E_n\cap [\varepsilon,1]$. $\eta\in ]\varepsilon,1]$ usually denotes another margin to obtain results that are uniform in the initial condition.
    \item $\tau_n=\inf\{t\geqslant 0\ \text{s.t. }\ m^n_t\leqslant\varepsilon\}$ is a stopping time defining a death time for the process $m^n$.
    \item For a random variable $m$, we define $\mathcal{L}(m)$ the probability distribution of $m$. Additionally, for an event $A$, we define  $\mathcal{L}(m\big|A)$ the law of $m$ conditionally on the event $A$. We similarly define $\mathbb{E}\left(\cdot\big| A\right)$ the conditional expectation.
    \item For a random process $m:[0,\infty[\mapsto[-1,1]$, we define $\mathbb{E}_{\nu}$ the expectation with respect to the process $m$ with initial distribution $m_0\sim\nu$. We similarly define $\mathbb{P}_{\nu}$. Whenever $\nu$ is a Dirac mass at a point $m$, we may, with a slight abuse of notations, write $\mathbb{E}_{m}=\mathbb{E}_{\delta_m}$.
    \item We write $\nu^n_t=\mathcal{L}\left(m^n_t\big|\tau_n>t\right)$ and denote $\nu^n_\infty$ the (unique) quasi-stationary distribution (i.e. the limit of $\nu^n_t$ for $t\rightarrow\infty$) defined in Theorem~\ref{thm:long-time-behavior}.
    \item For a function $f:[-1,1]\mapsto \mathbb{R}$, we define the norms
    \begin{align*}
        \|f\|_{\infty}=\sup_{x\in[-1,1]}|f(x)|\qquad \text{and }\qquad \|f\|_{\textrm{lip}}=\sup_{x,y\in[-1,1], x\neq y}\left|\frac{f(x)-f(y)}{x-y}\right|.
    \end{align*}
    \item $\mathcal{W}_1,\mathcal{W}_2,d_{TV}:\mathcal{P}(\mathbb{R})\times\mathcal{P}(\mathbb{R})\mapsto \mathbb{R}^+$ are respectively the Wasserstein-1, Wasserstein-2, and Total Variation distances defined in \eqref{eq:def_distances}. Similarly $d^{TV}_{\xi, V, E}:\mathcal{P}(\mathbb{R})\times\mathcal{P}(\mathbb{R})\mapsto \mathbb{R}^+$ is a weighted Total Variation distance given in \eqref{eq:def_TV_weighted}.
    \item For a probability measure $\nu\in\mathcal{P}([-1,1])$ and a function $f:[-1,1]\mapsto \mathbb{R}$ we write $\nu(f)=\int_{-1}^1f(x)d\nu(x)$. With a slight abuse of notation, whenever $\nu$ is defined on a finite space $E$, we may write for $m\in E$ the explicit value of the probability $\nu(m)$.
    \item For a semi-group $P$, we denote its effect on a function $f$ by $P_tf(x)=\mathbb{E}_x(f(m_t))$, and on a measure by
    $\nu P_t=\mathcal{L}(m_t)$ where $m_0\sim \nu$. Note that in particular, for $\nu=\delta_m$, we have $\delta_mP_t(f)=P_tf(m)$. 
    \item $M^n$ is the semi-group of the killed process, defined by $M^n_tf(m)=\mathbb{E}_m\left(f(m^n_t)\mathds{1}_{\tau_n>t}\right)$.
    \item $h_n:E^\varepsilon_n\mapsto ]0,1]$ is the right eigenvector of $M^n$ defined in Lemma~\ref{lem:existence-hn}. $b_n>0$ is a constant defining its eigenvalue. 
    \item $P^n$ is a conservative semi-group, the h-transform of $M^n$ defined in \eqref{eq:def_h_transform} using $h_n$ and $b_n$.
    \item $K_n\subset]\varepsilon,1]$ is a compact set and $\nu_n\in\mathcal{P}(E^\varepsilon_n)$ is a probability distribution, both defined in Lemma~\ref{lem:densite-CW}.
    \item $V_n:E^\varepsilon_n\mapsto \mathbb{R}_+$ is a Lyapunov function defined in Lemma~\ref{lem:lyapunov-CW-killed}.
    \item $(\mu^n_t)_{t\geqslant0}$ and $(\overline{\mu}_t)_{t\geqslant0}$ are two auxiliary processes defined in Section~\ref{sec:aux_proc}. $U$ is their underlying potential, and $\mathcal{B}_n$ and $\mathcal{B}$ their respective generators.
    \item Throughout the document, $C$ and $c$ (or possibly $C_i$ and $c_i$ for some $i\in\n$) are positive constants, that may change from one line to the next, and their exact values are of no interest to us. We simply insist on the fact that they are independent of $n$.
    \item $\mathds{1}$ can denote the constant function equal to 1.
\end{itemize}

%
%
%
%

\section{Some intermediate results}\label{sec:inter-result}

This section is devoted to some results that we need in order to prove Theorem~\ref{thm:long-time-behavior}. We first prove in Section~\ref{s-sec:qques_resultats_CW} some results on the non-conditioned Curie-Weiss process. Then, in Section~\ref{sec:eigentrucs}, we prove the existence of a left eigenvector $h_n$ for $M^n$, as well as estimates on the eigenvalue $b_n$, on $h_n$ and on the death time $\tau_n$.

%
%
%
%

\subsection{Results on the Curie-Weiss model}\label{s-sec:qques_resultats_CW}

We first prove Proposition~\ref{prop:PoCnonUnif}. 


\begin{proof}[Proof of Proposition~\ref{prop:PoCnonUnif}]
Recall the definition of the function $g$ given in~\eqref{eq:pot-CW}, and write
\[
f_k(m) = m^k,\qquad  k\in \left\{1,2\right\}.
\]
A direct computations yields that
\begin{align*}
\mathcal A_nf_1(m) =& \mathcal Af_1(m) = -g'(m),\\ \mathcal A_nf_2(m) =& \mathcal Af_2(m) + 4\frac{\lambda^n_+(m)+\lambda^n_-(m)}{n^2}= -2 g'(m)m + 4\frac{\lambda^n_+(m)+\lambda^n_-(m)}{n^2}\leqslant  -2g'(m)m +\frac{8e^{\beta}}{n}.
\end{align*}
Thus, we have, for any initial conditions $m^n_0,\overline{m}_0\in[-1,1]$
\begin{align*}
\frac{d}{dt}\E\left|m_t^n-\overline{m}_t\right|^2=&\frac{d}{dt}\left(\E\left|m_t^n\right|^2-2 \overline{m}_t\E m_t^n+(\overline{m}_t)^2\right)\\
\leqslant&-2 \E \left(g'(m_t^n)m_t^n\right)+2\E \left(g'(m_t^n)\overline{m}_t\right)+2\E \left(g'(\overline{m}_t)m_t^n\right)-2\E g'(\overline{m}_t)\overline{m}_t  + \frac{8e^{\beta}}{n}\\
=&-2\E \left((g'(m_t^n)-g'(\overline{m}_t))(m_t^n-\overline{m}_t)\right)+\frac{8e^{\beta}}{n}.
\end{align*}

\begin{description}
\item[$\bullet$]
Fix $\beta<1$. In this case 
\[
g''(m) = 2(1-\beta)\ch(\beta m) + 2\beta m\sh(\beta m) \geqslant 2(1-\beta),
\]
so that a convexity inequality yields
\[
\frac{d}{dt}\E\left|m_t^n-\overline{m}_t\right|^2 \leqslant -4(1-\beta) \E\left|m_t^n-\overline{m}_t\right|^2 +\frac{8e^{\beta}}{n}.
\]
A direct computation then yields 
\begin{align*}
\frac{d}{dt}\left[e^{4(1-\beta)t}\left(\E\left|m_t^n-\overline{m}_t\right|^2-\frac{2e^{\beta}}{(1-\beta)n}\right)\right]\leqslant0,
\end{align*}
and thus
\begin{align}\label{eq:couplage-poc-beta<1}
\E\left|m_t^n-\overline{m}_t\right|^2\leqslant e^{-4(1-\beta)t}\E\left|m_0^n-\overline{m}_0\right|^2+\frac{2e^{\beta}}{(1-\beta)n}.
\end{align}
Now fix an optimal coupling $(m^n_0,\overline m_0)$ of $\mathcal{L}(m_0^n)$ and $\mathcal{L}(\overline{m}_0)$:
\[
\mathcal{W}_2(\mathcal{L}(m_0^n),\mathcal{L}(\overline{m}_0))^2 = \E\left|m_0^n-\overline{m}_0\right|^2. 
\]
By definition of the Wasserstein distance, \eqref{eq:couplage-poc-beta<1} yields
\[
\mathcal{W}_2(\mathcal{L}(m_t^n),\mathcal{L}(\overline{m}_t))^2 \leqslant \E\left|m_t^n-\overline{m}_t\right|^2 \leqslant e^{-4(1-\beta)t}\mathcal{W}_2(\mathcal{L}(m_0^n),\mathcal{L}(\overline{m}_0))^2+\frac{2e^{\beta}}{(1-\beta)n}
\]
and concludes for the uniform in time propagation of chaos in the non-metastable case.

\item[$\bullet$] Fix now $\beta>1$. The potential $g$ is no longer convex, but it is still a smooth function on the compact interval $[-1,1]$, and thus there exists $C>0$ such that for all $m,m'\in[-1,1]$
\begin{align*}
|g'(m)-g'(m')|\leqslant C|m-m'|.
\end{align*}
Therefore
\begin{equation*}
\frac{d}{dt}\E\left|m_t^n-\overline{m}_t\right|^2\leqslant 2C\E\left|m_t^n-\overline{m}_t\right|^2+\frac{4e^{\beta}}{n},
\end{equation*}
and Gronwall's Lemma concludes the proof of propagation of chaos in the metastable case.

\item[$\bullet$] In addition, for $\beta>1$, let us show that we cannot improve this result. Fix $n\in\n$, $m_0^n=\overline m_0>0$, and write $f:m\mapsto m$. The process $(m^n_t)_t$ is an irreducible Markov process in a finite state space, hence admits a unique stationary distribution $\mu_{n,\infty}$ (which is a function defined on $E^n$). By symmetry, we know that $x\in E^n\mapsto\mu_{n,\infty}(x)$ is an even function, so that $\mu_{n,\infty}(f) = 0$. In particular, we may find $t_0>0$ such that
\[
\forall\,t\geqslant t_0,\qquad \left|\mathbb{E}f(m^n_t)\right|\leqslant  \frac{m_+}{4}.
\]
Note that a priori $t_0$ may depend on $n$, but this is not an issue. Moreover, we have that
\begin{align*}
f(\overline{m}_t)\xrightarrow[t\rightarrow\infty]{} f(m_+)=m_+\quad \text{ by continuity},
\end{align*}
so that we may find $t_1\geqslant t_0$ such that for all $t\geqslant t_1$, $\overline{m}_t\geqslant \frac{m_+}{2}$.
This way, for $t\geqslant t_1$
\begin{align*}
\mathcal{W}_2(\mathcal{L}(m^n_t),\delta_{\overline{m}_t})\geqslant \mathcal{W}_1(\mathcal{L}(m^n_t),\delta_{\overline{m}_t})\geqslant \left|\mathbb{E}f(m^n_t)-f(\overline{m}_t)\right|=  \left|\mathbb{E}\po m^n_t-\overline{m}_t\pf \right|
\geqslant  \frac{m_+}{4} =:\eta,
\end{align*}
which concludes the proof of~\eqref{eq:contrexemple_PoC_unif}.
\end{description}
\end{proof}

Let us now prove a local Doeblin condition for the Curie-Weiss process. This will be of use in the study of the long-time behavior of the sub-Markovian semi-group $M^n$.


\begin{lemma}\label{lem:densite-CW}
Fix $\omega>0$ and write 
\begin{align*}
K_n = \left[m_+-\frac{\omega}{\sqrt{n}},m_++\frac{\omega}{\sqrt{n}}\right]\cap E^\varepsilon_n.
\end{align*}
For all $t\geqslant 0$, there exists $c>0$ such that for all $n\in\mathbb{N}$, there exists a probability measure $\nu_n$ on $E_n^\varepsilon$ such that for all $m\in K_n$ and $\varepsilon<a<b<1$
\begin{align*}
\mathbb P_m\po m^n_t \in [a,b], \tau_n>t \pf\geqslant c \nu_n([a,b]),
\end{align*}
and $\nu_n(K_n)=1$.
\end{lemma}


\begin{proof}
Let us denote
\[
\overline K_n = \left[m_+-\frac{3\omega}{\sqrt{n}},m_++\frac{3\omega}{\sqrt{n}}\right]\cap E^\varepsilon_n,
\]
and
\[
\underline m_0^n = \min K_n ,\qquad \overline m_0^n = \max K_n.
\]
Up to changing $\overline m_0^n$ into $\overline m_0^n + 2/n$, we assume that $n(\overline m_0^n-\underline m_0^n)$ is a multiple of $4$ for all $n\in \N$.
For all $m\in K_n$, we construct a coupling $(\underline m_t^n,m_t^n,\overline m_t^n)$ such that:
\begin{enumerate}
\item $m^n$ is the Curie-Weiss process with initial condition $m$.
\item Almost surely, for all $0\leqslant t \leqslant S$
\begin{equation*}
\underline m_t^n \leqslant m_t^n \leqslant \overline m_t^n,
\end{equation*}
where 
\[
S = \inf\left\{t\geqslant 0, \underline m_t^n \notin \overline K_n\text{ or }\overline m_t^n \notin \overline K_n  \right\}.
\]
\item For all $s\leqslant t$
\[
\underline m_s^n = \overline m_s^n \implies \underline m_t^n = \overline m_t^n.
\]
\item For all $t >0$
\[
\liminf_{n\rightarrow\infty} \mathbb P\po \underline m_t^n = \overline m_t^n, S>t \pf  > 0.
\]
\end{enumerate}
Suppose that such a coupling exists. Then, for all $m\in K_n$ we have
\begin{align*}
    \mathbb P_m\po m^n_t \in [a,b], \tau_n>t \pf &\geqslant \mathbb P\po \underline m_t^n \in [a,b], \underline m_t^n = \overline m_t^n, S>t \pf \\ & \geqslant \mathbb P\po \underline m_t^n = \overline m_t^n, S>t \pf \mathbb P\left( \underline m_t^n \in [a,b]\  \Big|\ \underline m_t^n = \overline m_t^n,  S>t \right) \\ &\geqslant c\nu_n\po[a,b]\pf,
\end{align*}
which would conclude the proof with $c = \inf_{n} \mathbb P\po \underline m_t^n = \overline m_t^n, S>t \pf$, and, by construction, $\nu_n(K_n)=1$. To construct such a coupling, first notice that, using $\lambda^n_+(m_+) = \lambda^n_-(m_+)$ by definition of $m_+$, there exists $r>0$ such that for all $m\in \overline K_n$
\[
\lambda^n_{\inf} := n\frac{1+m_+}{2}e^{-\beta m_+}\po 1-\frac{r}{\sqrt{n}}\pf \leqslant \lambda_-^n(m),\lambda_+^n(m) \leqslant n\frac{1+m_+}{2}e^{-\beta m_+}\po 1+\frac{r}{\sqrt{n}}\pf =: \lambda^n_{\sup}.
\]
Then, for $t\leqslant S$, define the coupling as a Markov process on $E_n^3$ with jump rates, for $m_1<m_2<m_3$:
\begin{equation*}
\begin{array}{ll}
(m_1,m_1,m_1) \rightarrow \left(m_1 + 2/n,m_1 + 2/n,m_1 + 2/n\right) &: \lambda^{n}_+(m_1),\\
(m_1,m_1,m_1) \rightarrow \left(m_1 - 2/n,m_1 - 2/n,m_1 - 2/n\right) &: \lambda^{n}_-(m_1),\\
(m_1,m_3,m_3) \rightarrow \left(m_1 - 2/n,m_3 + 2/n,m_3 + 2/n\right) &: \lambda^{n}_+(m_3),\\
(m_1,m_3,m_3) \rightarrow \left(m_1 + 2/n,m_3 - 2/n,m_3 - 2/n\right) &: \lambda^{n}_{\inf},\\
(m_1,m_3,m_3) \rightarrow \left(m_1 ,m_3 - 2/n,m_3\right) &: \lambda^{n}_-(m_3)-\lambda^{n}_{\inf},\\
(m_1,m_3,m_3) \rightarrow \left(m_1 - 2/n,m_3 ,m_3 + 2/n\right) &: \lambda^{n}_{\sup}-\lambda^{n}_-(m_3),\\
(m_1,m_1,m_3) \rightarrow \left(m_1 - 2/n,m_1 - 2/n,m_3 + 2/n\right) &: \lambda^{n}_-(m_1),\\
(m_1,m_1,m_3) \rightarrow \left(m_1 + 2/n,m_1 + 2/n,m_3 - 2/n\right) &: \lambda^{n}_{\inf},\\
(m_1,m_1,m_3) \rightarrow \left(m_1 ,m_1 + 2/n,m_3 \right) &: \lambda^{n}_+(m_1) - \lambda^{n}_{\inf},\\
(m_1,m_1,m_3) \rightarrow \left(m_1 - 2/n,m_1 ,m_3 + 2/n\right) &: \lambda^{n}_{\sup} - \lambda^{n}_-(m_1),\\
(m_1,m_2,m_3) \rightarrow \left(m_1 - 2/n,m_2 ,m_3 + 2/n\right) &: \lambda^{n}_{\sup},\\
(m_1,m_2,m_3) \rightarrow \left(m_1 + 2/n,m_2 ,m_3 - 2/n\right) &: \lambda^{n}_{\inf},\\
(m_1,m_2,m_3) \rightarrow \left(m_1 ,m_2 - 2/n,m_3 \right) &: \lambda^{n}_-(m_2),\\
(m_1,m_2,m_3) \rightarrow \left(m_1 ,m_2 + 2/n,m_3 \right) &: \lambda^{n}_+(m_2).
\end{array}
\end{equation*}
The idea behind the coupling is that when $\underline m^n_t < m^n_t < \overline m^n_t$, $m^n_t$ is independent from $(\underline m^n_t, \overline m^n_t)$ and $(\underline m^n_t, \overline m^n_t)$ is a coupling by reflection of two jump processes, $\underline m^n$ (resp. $\overline m^n$) jumping to the right (resp. left) with rate $\lambda_{\inf}$ and  to the left (resp. right) with rate $\lambda_{\sup}$. When $m_t^n$ is equal to $\underline m_t^n$ or $\overline m_t^n$, they are coupled in order to keep the first two conditions, and finally when the three processes touch, they stay equal for all time to a process evolving as the Curie-Weiss dynamic. Condition~1,~2 and~3 are direct from the jump rates. Let us now show the fourth condition. Write 
\[
\Delta_t = \overline m_t^n - \underline m_t^n, \qquad T = \inf\left\{ t \geqslant 0, \Delta_t = 0\right\}.
\]
Because $(\underline m^n_t, \overline m^n_t)$ is a coupling by reflection and thus, until $\underline m^n_t=\overline m^n_t$, their mean $\frac{\overline m^n_t-\underline m^n_t}{2}$ is constant, we have that 
\[
S\mathds 1_{S<T} = \inf\left\{ t\geqslant 0, \Delta_t = \alpha_n\right\}\mathds 1_{S<T},
\]
for some $\alpha_n \approx 6\omega/\sqrt{n}$.
For all $t\leqslant S\wedge T$, $\Delta$ is a jump process on $\left\{0, 4/n, 8/n, \dots, \alpha_n - 4/n, \alpha_n\right\}$, with $\Delta_0 = \overline m^n_0 - \underline m^n_0 \approx 2\omega/\sqrt{n}$ and jump rates
\[
\Delta \rightarrow \Delta + 4/n: \, \lambda^{n}_{\sup},\qquad  \Delta \rightarrow \Delta - 4/n: \, \lambda^{n}_{\inf}.
\]
Condition~4 now reads
\[
\liminf_{n\rightarrow\infty} \mathbb P\po \tau_0 < t\wedge \tau_{\sqrt{n}\alpha_{n}} \pf >0,\qquad \tau_i = \inf\left\{t\geqslant 0, \, \sqrt{n}\tilde \Delta_t = i \right\},
\]
where $\po \tilde \Delta_t \pf_{t\geqslant 0}$ is a jump process on $\frac{4}{n}\N$, constructed in such a way that $\tilde \Delta_t = \Delta_t$ for all $0\leqslant t \leqslant T \wedge S$. 
Compute the generator of $\sqrt{n}\tilde \Delta$, denoted here and only here $\mathcal{G}_n$, for a function $f\in\mathcal{C}^\infty([0,1],\mathbb{R})$ (i.e. the set of smooth functions which is, in the sense of \cite[Theorem 19.10]{Kal02}, a core for the generators of both jump and diffusion processes) 
\begin{align*}
\mathcal{G}_n f(m)=&\lambda^{n}_{\sup}\left(f \left(m+\frac{4}{\sqrt{n}}\right)-f(m)\right)+\lambda^{n}_{\inf}\left(f\left(m-\frac{4}{\sqrt{n}}\right)-f(m)\right)\\
=&\frac{n}{2}(1+m_+)e^{-\beta m_+}\left(1+\frac{r}{\sqrt{n}}\right)\left(f\left(m+\frac{4}{\sqrt{n}}\right)-f(m)\right)\\
&+\frac{n}{2}(1+m_+)e^{-\beta m_+}\left(1-\frac{r}{\sqrt{n}}\right)\left(f\left(m-\frac{4}{\sqrt{n}}\right)-f(\sqrt{n}m)\right)\\
=&4r(1+m_+)e^{-\beta m_+}f'(m)+8(1+m_+)e^{-\beta m_+}f''(m)+ \underset{n\rightarrow\infty}{o}\left(1\right).
\end{align*}
From \cite[Theorem 19.25]{Kal02}, we obtain that
\begin{align*}
\mathcal Law\po\left(\sqrt{n}\tilde \Delta_s\right)_{0\leqslant s\leqslant t}\pf \xrightarrow[n\rightarrow \infty]{} \mathcal Law\po \left(X_s\right)_{0\leqslant s\leqslant t} \pf,
\end{align*}
where $\left(X_s\right)_{0\leqslant s\leqslant t}$ is the (unique strong) solution of the SDE
\begin{align*}
dX_t= 4r(1+m_+)e^{-\beta m_+}dt+4\sqrt{(1+m_+)e^{-\beta m_+}}dB_t,\qquad X_0= 2\omega,
\end{align*}
and $B$ is a standard Brownian motion.
Since $\sqrt{n}\alpha_n \rightarrow 6\omega$, Condition~4 is implied by 
\[
\liminf_{n\rightarrow\infty} \mathbb P\po \tau_0 < t\wedge \tau_{5\omega} \pf > 0.
\]
Additionally, the functions defined on $\mathcal C([0,T],\mathbb{R})$ by $x\mapsto \inf_{[0,T]}x$ and $x\mapsto \sup_{[0,T]}x$ are continuous. Hence, because the minimum and the maximum of a diffusion are atomless random variables, we have the convergence
\[
\mathbb P_{\Delta_0}\po \tau_0 < t\wedge \tau_{5} \pf  = \mathbb P_{\tilde \Delta_0}\po \inf_{[0,t]} \sqrt{n} \tilde \Delta < 0,\quad \sup_{[0,t]} \sqrt{n} \tilde \Delta < 5\omega \pf \xrightarrow[n\rightarrow\infty]{} \mathbb P_{2\omega}\po \overline\tau_0 < t\wedge \overline\tau_{5\omega} \pf >0,
\]
where
\[
\overline\tau_x = \inf\left\{t\geqslant 0, \, X_t = x \right\},
\]
which concludes the proof.
\end{proof}

This last lemma shows that $m_+$ is not a degenerate minima of $g$.

\begin{lemma}\label{lem:mini-non-degenere}
    We have $g''(m_+) > 0$.
\end{lemma}

\begin{proof}
By definition of $m_+$, we have
\[
(1+m_+)e^{-\beta m_+} = (1-m_+)e^{\beta m_+}.    
\]
Using this equality, we get
\begin{align*}
g''(m_+) =& \beta m_+ \po e^{\beta m_+} - e^{-\beta m_+} \pf + (1-\beta) \po e^{\beta m_+} + e^{-\beta m_+} \pf \\ 
=& e^{\beta m_+}\left[\beta m_+ \po 1- \frac{1-m_+}{1+m_+}\pf + (1-\beta)\po 1+ \frac{1-m_+}{1+m_+} \pf  \right] \\ 
=& \frac{2e^{\beta m_+}}{1 + m_+} \po 1 - \beta + \beta m_+^2 \pf.
\end{align*}
Thus, it only remains to show that $m_+ > \sqrt{1 - \frac{1}{\beta}}$, or equivalently that $m=\sqrt{1 - \frac{1}{\beta}}$ implies $g'(m) <0$, i.e. $\lambda_-^n(m) < \lambda_+^n(m)$. Denote
\begin{align*}
f(\beta) := \frac{\lambda_+^n\left(\sqrt{1 - \frac{1}{\beta}}\right)}{\lambda_-^n\left(\sqrt{1 - \frac{1}{\beta}}\right)} = \frac{1-\sqrt{1 - \frac{1}{\beta}}}{1 + \sqrt{1 - \frac{1}{\beta}}}e^{2\beta \sqrt{1 - \frac{1}{\beta}}} = \beta \po 1 - \sqrt{1 - \frac{1}{\beta}}\pf^2e^{2\beta\sqrt{1 - \frac{1}{\beta}}} .
\end{align*}
Note that $f(1) = 1$ and, for $\beta>1$, a standard computation yields
\begin{align*}
f'(\beta) = 2\beta e^{2\beta\sqrt{1- \frac{1}{\beta}}}\left(\sqrt{1- \frac{1}{\beta}}-1\right)^2\sqrt{1- \frac{1}{\beta}}>0,
\end{align*}
implying that $\lambda_+^n\left(\sqrt{1 - \frac{1}{\beta}}\right)>\lambda_-^n\left(\sqrt{1 - \frac{1}{\beta}}\right)$ for $\beta>1$, i.e. $m_+ > \sqrt{1 - \frac{1}{\beta}}$, which concludes the proof. 
\end{proof}

%
%
%
%

\subsection{Estimates on the eigen-elements}\label{sec:eigentrucs}

Recall the definition of the semi-group of the killed process
\[
M^n_tf(m)=\mathbb{E}_m\left(f(m^n_t)\mathds{1}_{\tau_n>t}\right).
\]
Since the goal is to study a Doob-transform of this semi-group, as explained in Section~\ref{s-sec:methode}, we need the existence of a right-eigenvector $h_n$. This is a consequence of Perron-Frobenius theorem. However, the existence in itself is not enough, and we show estimates on $h_n$ as well as on its corresponding eigenvalue.


\begin{lemma}\label{lem:existence-hn}
For all $n\in\mathbb{N}$, there exists a unique function $h_n:E_n^\varepsilon\to\R_+^*$, and a constant $b_n>0$, such that 
\begin{align*}
M^n_th_n=e^{-b_nt}h_n, \quad \|h_n\|_{\infty}=1.
\end{align*}
Additionally, we have that for all $n\in \N$, there exist a measure $\hat{\nu}^n_\infty$ on $E_n^\varepsilon$, and $\gamma_n>0$, such that for all $f:E_n^\varepsilon\to\R$ and all $x\in E^\varepsilon_n$
\begin{equation}\label{eq:conv-non-explicite}
\left| e^{b_n t}M_t^n(f)(x) - h_n(x)\hat{\nu}^n_\infty(f) \right| \leqslant Ce^{-\gamma_n t}\|f\|_{\infty}.
\end{equation}
\end{lemma}

The problem here is that Perron-Frobenius' theorem does not give any information on the spectral gap. In particular, one could a priori  have $\gamma_n\rightarrow 0$, and this would not yield the desired result. This is why we resort to an Harris-type theorem to get the more precise version~\eqref{eq:long-time-semi-group} below. Here, $\hat{\nu}^n_\infty$ is not a probability measure, but had we chosen the normalization $\hat{\nu}^n_\infty(h_n)=1$ instead of $\|h_n\|_{\infty}=1$, $\hat{\nu}^n_\infty$ would be the QSD.

\begin{proof}[Proof of Lemma~\ref{lem:existence-hn}]
Write 
\[
\varepsilon_n = \min E_n^\varepsilon, \qquad \overline \lambda^n = \sup_{m\in E^\varepsilon_n}\lambda_-^n(m) + \lambda_+^n(m).
\]
Let us denote $\Lambda^n$ the generator of $M^n_t$, namely the matrix defined by
\[
\Lambda^n_{m,m-\frac{2}{n}} = \lambda_-^n(m),\, m \neq \varepsilon_n,\qquad \Lambda^n_{m,m+\frac{2}{n}} = \lambda_+^n(m),\qquad \Lambda^n_{m,m} =-\left( \lambda_-^n(m) + \lambda_-^n(m)\right),
\]
and $\Lambda^n_{m,m'} = 0$ otherwise. The matrix $\Lambda^n$ corresponds to the generator of the semi-group $M^n$, and equal to $\mathcal A$ for all $m,m'>\varepsilon_n$. If $I_n$ denote the identity matrix, we have that $\Lambda^n + \overline \lambda^nI_n$ is a positive irreducible matrix, with all lines summing to at most  $\overline \lambda^n$. Hence, Perron-Frobenius yields that there exists 
\[
0< \rho_n\leqslant \overline \lambda^n, \qquad h_n:E^\varepsilon_n\to \R_+^*,\quad \text{s.t.}\quad (\Lambda^n + \overline \lambda^nI_n)h_n = \rho_nh_n,
\]
and $\rho_n$ corresponds to the spectral radius, and $h_n$ is unique up to normalization. In particular, writing $b_n = \overline \lambda^n - \rho_n$, we have that:
\[
\partial_t M_t^n(h_n) = M_t^n(\Lambda^nh_n) = -b_n M_t^n(h_n), 
\]
and $M_t^n(h_n) = e^{-b_n t}h_n$. Now because $\rho_n$ is the spectral radius of $\Lambda^n + \overline \lambda^nI_n$, all other eigenvalues $\rho$ of $\Lambda^n$ satisfy 
\[
\Re(\rho) < -b_n - \gamma_n,
\]
where $\Re(\rho)$ is the real part of $\rho$ and for some $\gamma_n>0$ corresponding to the spectral gap. In particular, classical finite dimensional algebra yields the convergence~\eqref{eq:conv-non-explicite}, and concludes the proof. 
\end{proof}

Let us now give two estimates. The first one is on the eigenvalue $b_n$, which is closely related to the killing rate (take $f=\mathds{1}$ in the convergence~\eqref{eq:conv-non-explicite}). More precisely, we show that $b_n\rightarrow 0$. The second  consists in showing that $h_n\rightarrow 1$ as $n\rightarrow \infty$. These two estimates are actually very natural, as already explained in Section~\ref{s-sec:methode}. Recall that the fact that $b_n$ goes to zero is equivalent to the fact that the time needed to reach $\varepsilon$ goes to infinity as $n$ goes to infinity. In other word, $M^n$ becomes a conservative Markov semi-group in the limit $n\rightarrow\infty$, and $h_n$ converges towards the associated left eigenvector of such a semi-group, which is the constant $\mathds{1}$.

\begin{lemma}\label{lem:estime-bn}
    For all $\delta > \varepsilon$, there exists $C_2,c>0$ such that for all $n\in\mathbb N$, there exists $C_1^n>0$, $m\in [\delta,1]\cap E_n$ and $t\geqslant 0$
    \[
    \mathbb{P}_m\left(\tau_n>t\right)\geqslant C^n_1\left(1-\frac{C_2}{\sqrt{n}}\right)^{ct}.
    \]
    Furthermore, $C^n_1\xrightarrow[n\rightarrow\infty]{}1$.
    As a consequence, it holds that
    \[
    \sup_{n\in\n}\sqrt{n}b_n < \infty.
    \]
\end{lemma}

\begin{proof}
$\bullet$ Let us first show that for all $m\in E_n^\varepsilon$
\begin{equation}\label{eq:taux-lim-ln-CW}
b_n = -\lim_{t\rightarrow \infty} \frac{1}{t}\ln\mathbb P_m\po \tau_n>t \pf.
\end{equation}
We have that $h_n>0$, hence $\inf_{E_n^\varepsilon}h_n>0$ and, since
\begin{align*}
h_n(m)=e^{b_n t}\mathbb{E}_m(h_n(m_t^n)\mathds{1}_{\tau_n>t}),
\end{align*}
we have on one hand
\begin{align*}
\inf_{E_n^\varepsilon}h_n \leqslant e^{b_nt}\mathbb P_m\po \tau_n>t \pf \sup_{E_n^\varepsilon}h_n 
\end{align*}
and on the other hand
\begin{align*}
\sup_{E_n^\varepsilon}h_n\geqslant e^{b_nt}\mathbb P_m\po \tau_n>t \pf \inf_{E_n^\varepsilon}h_n.
\end{align*}
This yields for all $t>0$
\begin{align*}
\frac{1}{t}\left[\ln\mathbb P_m\po \tau_n>t \pf +\ln\frac{\inf_{E_n^\varepsilon}h_n}{\sup_{E_n^\varepsilon}h_n} \right] \leqslant -b_n\leqslant \frac{1}{t}\left[\ln\mathbb P_m\po \tau_n>t \pf +\ln\frac{\sup_{E_n^\varepsilon}h_n}{\inf_{E_n^\varepsilon}h_n} \right].
\end{align*}
Therefore
\begin{align*}
 \limsup_{t\rightarrow \infty} \frac{1}{t}\ln\mathbb P_m\po \tau_n>t \pf\leqslant -b_n \leqslant  \liminf_{t\rightarrow \infty} \frac{1}{t}\ln\mathbb P_m\po \tau_n>t \pf,
\end{align*}
and we thus obtain the convergence~\eqref{eq:taux-lim-ln-CW}. 

$\bullet$ Now fix $\eta >0$ and $t_0>0$ such that 
\[|\overline m_0-m_+|<\eta \implies |\overline m_{t_0}-m_+|<\eta/2,\qquad \varepsilon < m_+ - 3\eta/2.\]
Write
\[p_n := \sup_{|m_0-m_+|<\eta}\mathbb P_{m_0}\po \sup_{0\leqslant t\leqslant t_0}|m_t^n-\overline m_t|\geqslant\eta/2 \pf,\]
and
\[
\mathcal{E}(t_1,t_2) := \left\{m^n \text{ s.t. }|m^n_{t_2}-m_+|<\eta,\text{ and }\forall\,t_1\leqslant t\leqslant t_2, m_t^n >\varepsilon \right\}.
\]
In particular,  
\[
\cap_{i\leqslant t/t_0} \mathcal{E}(i t_0,(i+1) t_0) \subset \left\{\tau_n > t \right\}.
\]
Notice that for all $t_1,t_2\geqslant 0$, if $m^n$ is such that $\left|m_{t_1}^n-m_+\right|<\eta$, then
\[
\left\{ m^n \text{ s.t. }\sup_{t_1\leqslant t\leqslant t_2} |m_t^n-\overline m_t|<\eta/2\right\} \subset \mathcal{E}(t_1,t_2),
\]
where $\overline m$ has initial condition $\overline m_{t_1}= m^n_{t_1}$,
so that we have for all $m\in E_n^\varepsilon$ such that $\left| m - \eta \right| < \eta$
\[
\mathbb P_m\po \tau_n > t \pf \geqslant (1-p_n)^{t/t_0}.
\]
In particular, the identity~\eqref{eq:taux-lim-ln-CW} yields
\begin{equation}\label{eq:borne-pn-bn}
b_n \leqslant \frac{p_n}{t_0(1-p_n)}. 
\end{equation}

$\bullet$
Let us show that there exists $c>0$ such that $p_n\leqslant c/\sqrt{n}$. Write, for $(\overline m_t)_t$ such that $\overline m_0=m_0^n$
\[
\mathcal M_t = m_t^n -m_0^n+ \int_0^t g'(m_s^n) \dd s = m_t^n -\overline m_t + \int_0^t \po g'(m_s^n) - g'(\overline m_s) \pf \dd s.
\]
The process $\mathcal M$ is a martingale, and because $g$ is smooth on $[0,1]$, Doob's martingale inequality yields that for any $r>0$
\begin{align*}
\mathbb P\po \sup_{0\leqslant t\leqslant t_0} \left| \mathcal M_{t} \right| > r \pf  \leqslant& C \E\po \left| \mathcal M_{t_0} \right| \pf \leqslant C\E\po \left| m_{t_0}^n - \overline m_{t_0}\right| + \|g''\|_{\infty}\int_0^{t_0} \left| m_{t}^n - \overline m_{t}\right|\dd t \pf \\
\leqslant& C(1+\|g''\|_{\infty}) \po  \sqrt{\E\po \po m_{t_0}^n - \overline m_{t_0} \pf^2 \pf} + \po \int_0^{t_0} \sqrt{\E\po \left( m_{t}^n - \overline m_{t}\right)^2\pf}\dd t \pf \pf \\
\leqslant& \frac{C(1+\|g''\|_{\infty})}{\sqrt{n}}
\end{align*}
for some $C>0$ that is independent of $n\in\n$ (but not of $t_0$ and $r$), and where we used Proposition~\ref{prop:PoCnonUnif} for the last inequality. Now, for all $0\leqslant t\leqslant t_0$, we have
\[
\left| m_{t}^n - \overline m_{t}\right| \leqslant \|g''\|_\infty \int_0^t \sup_{0\leqslant u\leqslant s} \left| m_{u}^n - \overline m_{u}\right| \dd s + \left| \mathcal M_t \right|,
\]
so that 
\[
\sup_{0\leqslant u \leqslant t} \left| m_{u}^n - \overline m_{u}\right| \leqslant \|g''\|_\infty\int_0^t \sup_{0\leqslant u\leqslant s} \left| m_{u}^n - \overline m_{u}\right| \dd s + \sup_{0\leqslant u \leqslant t_0} \left| \mathcal M_u \right|,
\]
and Grönwall's Lemma yields that there exists $C>0$ such that
\[
\sup_{0\leqslant u \leqslant t} \left| m_{u}^n - \overline m_{u}\right| \leqslant C\sup_{0\leqslant u \leqslant t_0} \left| \mathcal M_u \right|.
\]
This yields that:
\[
p_n \leqslant \mathbb P\po \sup_{0\leqslant t\leqslant t_0} \left| \mathcal M_{t} \right| > \eta/2C \pf \leqslant C/\sqrt{n},
\]
which concludes this point.

$\bullet$ We can now conclude the proof. The bound $\sup_{n\in\n}\sqrt{n}b_n < \infty$ is a direct consequence of the previous inequality and~\eqref{eq:borne-pn-bn}.

To get the bound on the time of death after time $t>0$, fix $\delta >\varepsilon$ and $t_1>0$ such that if $\overline{m}_0 \geqslant \delta$, $|\overline{m}_{t_1}-m_+|< \eta/2$. Using Markov property, we get that for all $m\geqslant \delta$ and $t\geqslant t_1$
\[
\mathbb P_m\po \tau_n > t \pf \geqslant \mathbb P_m \po \sup_{0\leqslant s \leqslant t_1} \left| m^n_{s} - \overline{m}_{s}\right| < \eta/2 \pf \inf_{|m'-m_+|<\eta} \mathbb P_{m'}\po \tau_n > t-t_1 \pf \geqslant \left(1-\frac{C}{\sqrt{n}}\right)(1-p_n)^{(t-t_1)/t_0},
\]
which concludes the proof (recall that $t_1, t_0$ are independent of $n$). 
\end{proof}

Now, let us turn to the estimates on $h_n$.

\begin{lemma}\label{lem:estime-hn}
    For all $\delta >\varepsilon$,  we have
    \[
    \inf_{[\delta,1]\cap E_n}h_n\xrightarrow[n\rightarrow \infty]{}1.
    \]
\end{lemma}


\begin{proof}
First notice that by taking $f=\mathds{1}$ in the convergence~\eqref{eq:conv-non-explicite}, we get that
\begin{equation*}
 h_n (m) = \hat{\nu}^n_\infty(\mathds{1})^{-1}\lim_{t \rightarrow +\infty} e^{b_n t} \mathbb{P}_m (\tau_n>t).
\end{equation*}
In particular, this implies that $h_n$ is non-decreasing on $E^\varepsilon_n$. We start by defining $h_n$ no longer only on $E^\varepsilon_n$ but on the entire interval $[\varepsilon,1]$ (and also denote it $h_n$ with a slight abuse of notations) by considering a simple linear interpolation of the points given in $E^\varepsilon_n$. Write now
\[
\Delta_n (m) := n \left[ h_n \left( m + \frac{2}{n}\right) - h_n ( m ) \right].
\]
Fix $\delta>0$. The proof now follows three steps.
\begin{enumerate}
    \item In the first step, we show that $\sup_{n\in\mathbb N}\sup_{\delta \leqslant m\leqslant m_+}\Delta(m) <\infty$.
    \item In the second step, we prove that $\sup_{n\in\mathbb N}\sup_{m_+\leqslant  m\leqslant 1}\Delta(m) <\infty$.
    \item From those two bounds, we deduce that the sequence $(h_n|_{[\delta,1]})$ is relatively compact. The third and last point consists in showing that there can be only one accumulation point, the constant $1$ function.
\end{enumerate}  

\emph{Step~1.} The fact that $\Lambda_nh_n = -b_nh_n$ (recall $\Lambda_n$ is the generator of $M^n$) on $[\delta,1]\cap E_n$ reads
\begin{equation}\label{eq:Delta}
r_\beta ( m) \Delta_n (m) + \frac{nb_n}{\lambda_-^n(m)} h_n (m) = \Delta_n \left(m -\frac{2}{n}\right). 
\end{equation}
where for all $m\in E^\varepsilon_n$
\[
r_\beta (m) := \frac{\lambda_+^n(m)}{\lambda_-^n(m)} = \frac{1-m}{1+m} e^{2 \beta m}.
\]
By definition of $m_+$, we have that $r_\beta\geqslant1$ on $[\varepsilon,m_+]$, with equality only on $m_+$, so that:
\begin{equation}\label{eq:order-delta}
\Delta_n(m) \leqslant r_\beta(m) \Delta_n(m) \leqslant \Delta_n\po m-\frac{2}{n}\pf,
\end{equation}
for all $m \in E^\varepsilon_n\cap [\varepsilon,m_+]$. Thus, for the first step, we only need to show that $(\Delta_n(\delta_n))_n$ is a bounded sequence, where $\delta_n = \inf \left\{ m\in E^\varepsilon_n,\, m\geqslant \delta\right\}$. Write
\[
r_0 := \min_{[\varepsilon,\delta]}r_\beta(m) >1,
\]
so that iterating~\eqref{eq:order-delta} yields for $\varepsilon_n=\min E^\varepsilon_n\in[\varepsilon,\varepsilon+\frac{2}{n}]$ 
\[
\Delta_n(\delta_n) \leqslant r_0^{-(\delta-\varepsilon)n/2-1}\Delta_n(\varepsilon_n) \leqslant Cnr_0^{-(\delta-\varepsilon)n/2} \xrightarrow[n\rightarrow \infty]{} 0,
\]
where we used $\|h_n\|_{\infty}=1$ (with $h_n\geqslant 0$) and thus $\|\Delta\|_{\infty}\leqslant n$, which concludes the first point.

\emph{Step~2.} First, Lemma~\ref{lem:estime-bn} and the identity $\Lambda_nh_n = -b_nh_n$ taken at $m=1$ yields that 
\[\Delta_n\po 1-\frac{2}{n} \pf = \frac{nb_n}{\lambda_-^n(1)}\leqslant \frac{C}{\sqrt{n}},\]
for some $C>0$. Next, let us show that
\begin{equation}\label{eq:prop-r}
    r'_\beta ( m_+) < 0,\qquad  r_\beta ( m ) \leqslant 1 + r'_\beta ( m_+) ( m - m_+),\quad \forall m \in [m_+, 1].
\end{equation}
We have 
\[
r_\beta(m) = 1-\frac{ng'(m)}{\lambda_-^n(m)},
\]
so that, since $g'(m_+)=0$ and using Lemma~\ref{lem:mini-non-degenere}, we have
\[
r'_\beta ( m_+) = -\frac{ng''(m_+)}{\lambda_-^n(m_+)} < 0.
\]
In fact, we even obtain $r'_\beta ( m)\leqslant 0$ for all $m \in [m_+, 1]$. Moreover, a direct computation yields
\[
r_\beta'(m) = 2\beta r_\beta(m) - \frac{2e^{2\beta m}}{(1+m)^2},
\]
which in turn yields
\[
r''(m) = 2\beta r_\beta'(m) + \frac{4e^{2\beta m}}{(1+m)^3} - \frac{4\beta e^{2\beta m}}{(1+m)^2}.
\]
Since $\beta \geqslant 1$ and $0\leqslant m\leqslant 1$ we get
\[
\frac{4e^{2\beta m}}{(1+m)^3} \leqslant \frac{4\beta e^{2\beta m}}{(1+m)^2},
\]
and, since $r_\beta'\leqslant 0$ on $[m_+,1]$, this concludes the proof of~\eqref{eq:prop-r}.
Write $\underline \lambda = \inf_{[m_+,1]}\lambda^n_-/n$. Iterating~\eqref{eq:Delta} yields for all $m\geqslant m_+$
\begin{align*}
\Delta_n (m) \leqslant \frac{b_n}{\underline \lambda} + \frac{b_n}{\underline \lambda}\sum_{m<m_1<1-\frac{2}{n}}\prod_{m< m_2\leqslant m_1} r_\beta ( m_2) + \Delta_n \po 1- \frac{2}{n}\pf \prod_{m\leqslant m_2\leqslant 1-\frac{2}{n}} r_\beta ( m_2),
\end{align*}
All sums and products are implicitly on $m\in E^\varepsilon_n$.
Since $r_\beta\leqslant 1$ on $[m_+,1]$, we have that
\[
\Delta_n \po 1- \frac{2}{n}\pf \prod_{m\leqslant m_2\leqslant 1-\frac{2}{n}} r_\beta ( m_2) \leqslant \frac{C}{\sqrt{n}} \xrightarrow[n\rightarrow \infty]{} 0.
\]
Then using that $r_\beta'(m_+)<0$ and inequality~\eqref{eq:prop-r}, we have that for all $m_+\leqslant m\leqslant m_1 \leqslant 1$ 
\begin{multline*}
    \prod_{m< m_2\leqslant m_1} r_\beta ( m_2)\leqslant \prod_{m< m_2\leqslant m_1}  \left(1 + r'_\beta ( m_+) \left(m_2 - m_+ \right) \right)
    \leqslant \exp\po r'_\beta(m_+)  \sum_{m< m_2\leqslant m_1} \left(m_2-m_+\right) \pf \\ \leqslant \exp\po r'_\beta(m_+) \sum_{m< m_2\leqslant m_1} (m_2-m) \pf = \exp\po \frac{2r'_\beta(m_+)}{n} \sum_{i=0}^{k_{m_1,m}} i \pf \leqslant \exp\po \frac{r'_\beta(m_+)k_{m_1,m}^2}{n} \pf,
\end{multline*}
where 
\[
k_{m_1,m} = n\frac{m_1-m}{2}.
\]
Therefore, 
\begin{equation*}
\frac{b_n}{\underline \lambda}\sum_{m<m_1<1-\frac{2}{n}}\prod_{m< m_2\leqslant m_1} r_\beta ( m_2)
\leqslant \frac{b_n}{\underline \lambda} \sum_{m < m_1 < 1-\frac{2}{n} }\exp\po \frac{r'_\beta(m_+)k_{m_1,m}^2}{n} \pf  \leqslant \frac{C}{\sqrt{n}}\sum_{k=1}^ne^{-\alpha k^2/n},
\end{equation*}
for some $C,\alpha>0$, where we simply made explicit the values of the spins $m\in E_n^\varepsilon$. We are left to show that this last quantity is bounded uniformly in $n\in\mathbb N$. Let us divide this sum into two parts. We first have
\[
\frac{1}{\sqrt{n}}\sum_{k=1}^{\sqrt{n}}e^{-\alpha k^2/n} \leqslant \frac{1}{\sqrt{n}} \times \sqrt{n}= 1.
\]
Secondly
\[
\frac{1}{\sqrt{n}}\sum_{k=\sqrt{n}}^{n} e^{-\alpha k^2/n} \leqslant \frac{1}{\sqrt{n}}\sum_{k=\sqrt{n}}^{n} (e^{-\alpha/\sqrt{n}})^k \leqslant \frac{1}{\sqrt{n}}\sum_{i=0}^{\infty} (e^{-\alpha/\sqrt{n}})^k \leqslant \frac{1}{\sqrt{n}}\frac{1}{1-e^{-\alpha/\sqrt{n}}}\xrightarrow[n\rightarrow \infty]{} \frac{1}{\alpha},
\]
which concludes this point.

\emph{Step~3.} From step~1 and~2, we get that the sequence $(h_n|_{[\delta,1]})$ is uniformly Lipchitz. In particular, Arzelà–Ascoli theorem yields that this sequence is a relatively compact set for the uniform norm. Let $h$ be an accumulation point, and let us show that $h=\mathds{1}$. Up to extraction, let us assume that $h_n\rightarrow h$. We have for all $t\geqslant 0$ and all $m^n_0=\overline{m}_0\geqslant \delta$
\begin{align*}
&\left| h(\overline m_t) - M_t^nh_n(m^n_0) \right| \\ &\qquad\qquad  \leqslant \left| h(\overline m_t) - h_n(\overline m_t) \right| + \left| h_n(\overline m_t) - \E_{m_0^n}\po h_n(m^n_t)\pf  \right| + \left| \E_{m_0^n}\po h_n(m^n_t)\pf  - M_t^nh_n(m^n_0) \right| \\ &\qquad \qquad\leqslant  \sup_{[\delta,1]}\left| h-h_n \right| + \|h_n\|_{lip}\sqrt{\E\po\po \overline m_t - m^n_t \pf^2\pf } + \mathbb P_{\delta_n}\po \tau_n \leqslant  t \pf \xrightarrow[n\rightarrow \infty]{}0,
\end{align*}
where we used Proposition~\ref{prop:PoCnonUnif} for the second term and Lemma~\ref{lem:estime-bn} for the third one. In particular, letting $n$ go to infinity in the equality $M_t^n h_n = e^{-b_n t}h_n$ yields that for all $t\geqslant 0$ and $\overline m_0\in [\delta,1]$
\[
h(\overline m_t) = h(\overline{m}_0).
\]
For all $\delta \leqslant m <m' <m_+$, there exists $t>0$ such that $\overline m_t = m'$ with initial condition $\overline m_0 = m$, so that $h$ is constant on $[\delta,m_+[$. The same holds on $]m_+,1]$, and since $h$ is continuous, we get that $h=\|h\|_{\infty} = 1$. Since $(h_n|_{[\delta,1]})$ is a relatively compact sequence with a unique accumulation point, it converges towards this accumulation point, which concludes the proof.
\end{proof}

%
%
%
%

\section{Proof of Theorem~\ref{thm:long-time-behavior} }\label{sec:cv_vers_qsd}

This section is devoted to the proof of Theorem~\ref{thm:long-time-behavior}, based on the result of Section~\ref{sec:inter-result}. To do so, define the $h$-transform of the sub-Markovian semi-group $M^n$ by
\begin{equation}\label{eq:def_h_transform}
P^n_tf=e^{b_nt}h_n^{-1}M^n_t(h_nf).
\end{equation}
This way, $P^n$ gets its semi-group properties from $M^n$, and is Markovian since, denoting $\mathds{1}$ the constant function equal to 1, we have 
\begin{align*}
P^n_t\mathds{1}=e^{b_nt}h_n^{-1}M_t^n(h_n)=\mathds{1}.
\end{align*}
This $h-$transform is therefore a Markov semi-group and, as a consequence, its long-time behavior can be studied using the classical methods developed by Meyn and Tweedie, namely through the existence of a Lyapunov function (Lemma~\ref{lem:densite-CW-killed} below) as well as a local Doeblin condition (Lemma~\ref{lem:lyapunov-CW-killed} below).

%
%
%
%

\subsection{The Doeblin and Lyapunov conditions}\label{s-sec:Lyapu-density-semi-group}

We provide here the necessary conditions to apply Harris theorem on the $h-$transform. The first one is the creation of density, which is a direct consequence of Lemma~\ref{lem:densite-CW}. Recall
\[
\delta_mP^n_\tau(m') := P^n_{\tau}\po \mathds 1_{\{m'\}} \pf (m).
\]


\begin{lemma}\label{lem:densite-CW-killed}
Let $\omega>0$ and recall from Lemma~\ref{lem:densite-CW}  the definition
\begin{align*}
K_n = \left[m_+-\frac{\omega}{\sqrt{n}},m_++\frac{\omega}{\sqrt{n}}\right]\cap E^\varepsilon_n.
\end{align*}
For all $\tau> 0$, there exist $c>0$ and $n_0\in\mathbb{N}$ such that for all $n\geqslant n_0$, there exists a probability measure $\nu_n$ on $E^\varepsilon_n$ such that for any $m'\in E^\varepsilon_n$ 
\begin{equation*}
    \inf_{m\in K_n}\delta_mP^n_\tau(m')\geqslant c \nu_n(m').
\end{equation*}
\end{lemma}


\begin{proof}[Proof of Lemma~\ref{lem:densite-CW-killed}]
Fix $\tau>0$. By definition, 
\begin{align*}
\delta_mP^n_\tau(m')=P^n_\tau\mathds{1}_{\{m'\}}(m)=e^{b_n\tau}h_n^{-1}(m)M^n_\tau(h_n\mathds{1}_{\{m'\}})(m).
\end{align*}
Lemma~\ref{lem:densite-CW} yields that there exists $c>0$ such that for all $n\in\mathbb{N}$ there exists $\tilde{\nu}_n$ such that for all $m,m'\in K_n$
\begin{equation*}
M^n_\tau(h_n\mathds{1}_{\{m'\}})(m) = h_n(m')\mathbb P_m\po m_\tau^n = m',\tau_n > \tau \pf \geqslant c \tilde \nu_n(m').
\end{equation*}
Therefore, using that $b_n\geqslant 0$ and $h_n\leqslant 1$
\begin{equation}\label{eq:int_densite_transform_0}
\delta_mP^n_\tau(m')\geqslant ce^{b_n\tau}h_n^{-1}(m)h_n(m')\tilde \nu_n(m')\geqslant ch_n(m')\tilde \nu_n(m') = c\tilde\nu_n(h_n) \nu_n(m'),
\end{equation}
where $\nu_{n}$ is a probability measure defined by
\begin{equation}\label{eq:int_densite_transform_1}
\nu_n(m'):=\frac{h_n(m')\tilde \nu_n(m')}{\tilde\nu_n(h_n)}.
\end{equation}
To conclude the proof, it only remains to show that $\liminf_n \nu_n(h_n) >0$. Fix $\varepsilon < \delta < m_+$. Lemma~\ref{lem:estime-hn} yields
\[
\inf_{m\in[\delta,1]}h_n(m)\rightarrow1.
\]
In particular, since $h_n(m)>0$ in $[\delta,1]$, there exists $\tilde{c}>0$ independent of $n$ such that $h_n\geqslant \tilde{c}\mathds{1}_{[\delta,1]}$. For $n$ large enough, $K_n\subset [\delta,1]$, and thus
\begin{equation}\label{eq:int_densite_transform_2}
\tilde \nu_n\left(h_n\right)\geqslant \tilde{c}\tilde \nu_n([\delta,1])\geqslant \tilde{c}\tilde \nu_n(K_n)=\tilde{c},
\end{equation}
which concludes the proof.
\end{proof}

The second condition is the existence of a Lyapunov function, which we define using the underlying potential $g$ (recall its definition from~\eqref{eq:pot-CW}).


\begin{lemma}\label{lem:lyapunov-CW-killed}
    For any $\tau>0$, there exist $b>0$, $0<a<1$, and $n_0\in\N$ (depending only on $\tau$, $\beta$ and $\varepsilon$) such that for all $n\geqslant n_0$ and $m\in E^\varepsilon_n$
    \begin{equation*}
    P^n_\tau V_n(m) \leqslant aV_n(m) + \frac{b}{n},
    \end{equation*}
    with $V_n:=h_n^{-1}g:E^\varepsilon_n\mapsto \mathbb{R}_+$.
\end{lemma}


\begin{proof}[Proof of Lemma~\ref{lem:lyapunov-CW-killed}]
We have that $\A g = -(g')^2$. Both the drift $g'$ and the potential $g$ have a unique zero in $[\varepsilon,1]$, which is $m_+$, and from Lemma~\ref{lem:mini-non-degenere}, $g''(m_+)> 0$. 
Therefore
\begin{align*}
\frac{(g'(m))^2}{g(m)}\xrightarrow[m\rightarrow m_+]{}2g''(m_+)>0,
\end{align*}
and thus, by continuity on $[\varepsilon,1]$, there thus exists $\gamma>0$ such that for all $m\in[\varepsilon,1]$
\begin{align*}
\frac{(g'(m))^2}{g(m)}\geqslant \gamma.
\end{align*}
This yields 
\begin{align*}
\A g(m) \leqslant -\gamma g(m).
\end{align*}
Now, for all $\C^2$ function $f$, we have from a Taylor expansion
\[
\|\A f-\A_n f\|_{\infty} \leqslant \frac{4e^{\beta}\|f''\|_{\infty}}{n},
\]
and thus for $m\in E_n^\varepsilon$ and since $g\geqslant0$, 
\[
\Lambda^ng(m) \leqslant  \A_n g(m) \leqslant \A g(m) + \frac{4e^{\beta}\|g''\|_{\infty}}{n} \leqslant -\gamma g(m) + \frac{4e^{\beta}\|g''\|_{\infty}}{n},
\]
where recall that $\Lambda^n$ is the generator of the semi-group $M^n$.
From Kolmogorov equation we then obtain
\begin{equation}\label{eq:lya_non_cons}
    M_{\tau}^n g (m) \leqslant e^{-\gamma\tau}g(m) + \frac{4e^{\beta}\|g''\|_{\infty}}{\gamma n} .
\end{equation}
Consider now $\varepsilon<m_0<m_+$ and, given $\tau>0$, consider $n_0\in\n$ such that for $n\geqslant n_0$ we have
\begin{align*}
\frac{4e^{\beta}\|g''\|_{\infty}}{\gamma n}\leqslant (e^{-\frac{\gamma }{2} \tau}-e^{-\gamma  \tau})g(m_0).
\end{align*}
Since for $\varepsilon<m<m_0$ we have $g(m)\geqslant g(m_0)$ (as, recall, for $\beta>1$ and $m\in[\varepsilon,m_+]$ we have $g'(m)\leqslant0$), plugging the above estimate back into \eqref{eq:lya_non_cons}, we get that for $n\geqslant n_0$
\begin{align*}
M_{\tau}^n g  (m)\leqslant e^{-\frac{\gamma \tau}{2}}g(m) + \frac{4e^{\beta}\|g''\|_{\infty}}{\gamma n}\mathds{1}_{m\in [m_0,1]} .
\end{align*}
Going back to the Doob-transform yields
\begin{equation*}
    P^n_\tau (h_n^{-1}g)(m) \leqslant e^{-\frac{\gamma \tau}{2}} e^{b_n\tau}  (h_n^{-1}g)(m) + e^{b_n\tau}h_n^{-1}(m)\frac{4e^{\beta}\|g''\|_{\infty}}{\gamma n}\mathds{1}_{m\in [m_0,1]}.
\end{equation*}
Lemmas~\ref{lem:estime-bn}~and~\ref{lem:estime-hn} yield $b_n\xrightarrow[n\rightarrow \infty]{}0$ and $\inf_{m\in[m_0,1]\cap E_n} h_n(m)\xrightarrow[n\rightarrow \infty]{} 1$ respectively, so that there exists $n_1\geqslant n_0$ such that for $n\geqslant n_1$
\begin{equation*}
    P^n_\tau (h_n^{-1}g)(m) \leqslant e^{-\frac{\gamma \tau}{4}} (h_n^{-1}g)(m) + \frac{8e^{\beta}\|g''\|_{\infty}}{\gamma n}=: a(h_n^{-1}g)(m) + \frac{b}{n}.
\end{equation*}
We now consider $V_n=h_n^{-1}g$ which, since we chose $\varepsilon\in\mathbb{R}\setminus\mathbb{Q}$ to have $h_n>0$ on $E_n\cap[\varepsilon,1]$, is well defined and satisfies for all $m\in[\varepsilon, 1]$, $V_n(m)\geqslant 0$ with equality if and only if $m=m_+$. This yields the result.
\end{proof}

%
%
%
%

\subsection{Contraction of the associated semi-group}\label{s-sec:proof-theorems}

We may now prove Theorem~\ref{thm:long-time-behavior} using Lemmas~\ref{lem:densite-CW-killed}~and~\ref{lem:lyapunov-CW-killed}. Define for a given $\xi>0$ and a nonnegative function $V$ on a space $E$
\begin{align*}
\|f\|_{L^\infty(\xi,V,E)}=\sup_{x\in E}\frac{|f(x)|}{1+ \xi V(x)},
\end{align*}
as well as the weighted total variation distance
\begin{equation}\label{eq:def_TV_weighted}
d^{TV}_{\xi,V,E}(\mu,\nu)=\sup_{f:\|f\|_{L^\infty(\xi,V,E)}\leqslant1}\int_E f(x)(\mu-\nu)(dx).
\end{equation}
Both notations are quite heavy, as we choose to insist on the parameters $\xi$, $V$ and $E$. We do so since, in the way we use them, they depend on $n$ and tracking this dependence on the number of particles is of major importance in this work. Note also that we have for all $\mu,\nu\in\mathcal{P}(E)$
\begin{align*}
    d_{TV}(\mu,\nu)\leqslant d^{TV}_{\xi,V,E}(\mu,\nu)\leqslant 2+ \xi \left(\mu(V)+\nu(V)\right).
\end{align*}
We use the following result (adapted and written in the finite case for simplicity) from \cite{HM11}.


\begin{proposition}[Theorems 1.2 and 1.3 from \cite{HM11}]\label{prop:yathhm}
    Let $\mathcal{P}$ be a Markov transition kernel on a finite space $\mathds{X}$. Assume
    \begin{itemize}
        \item There exists a function $V:\mathds{X}\mapsto [0,\infty[$ and constants $K\geqslant0$ and $\gamma\in]0,1[$ such that for all $x\in \mathds{X}$
        \begin{align*}
        \mathcal{P}V(x)\leqslant \gamma V(x)+K.
        \end{align*}
        \item There exists a constant $\alpha\in]0,1[$ and a probability measure $\nu$ such that for all $y\in \mathds{X}$
        \begin{align*}
        \inf_{x\in\mathcal{C}} \delta_x\mathcal{P}(y)\geqslant \alpha \nu(y),
        \end{align*}
        with $\mathcal{C}=\left\{x\in\mathds{X}\ :\ V(x)\leqslant R\right\}$ for some $R>\frac{2K}{1-\gamma}$.
    \end{itemize}
    Then $\mathcal{P}$ admits a unique invariant measure $\mu^*$. Furthermore, denoting
    \begin{align*}
    \alpha_0\in]0,\alpha[,\qquad \gamma_0\in\left]\gamma+\frac{2K}{R},1\right[,\qquad \xi=\frac{\alpha_0}{K},\qquad \bar{\alpha}=\max\left(1-(\alpha-\alpha_0),\frac{2+R\xi\gamma_0}{2+R\xi}\right),
    \end{align*}
    we have for any probability measure $\mu,\nu$ on $\mathds{X}$
    \begin{align*}
    d^{TV}_{\xi,V,\mathds{X}}(\mu \mathcal{P},\nu \mathcal{P})\leqslant \bar{\alpha}d^{TV}_{\xi,V,\mathds{X}}(\mu,\nu).
    \end{align*}
\end{proposition}

We may now prove Theorem~\ref{thm:long-time-behavior}.


\begin{proof}[Proof of Theorem~\ref{thm:long-time-behavior}]
Fix $\tau>0$.

\textbf{$\bullet$ Applying Proposition~\ref{prop:yathhm}.} 
Let us see that $P^n_\tau$ satisfies the assumptions of Proposition~\ref{prop:yathhm}. The first assumption is obtained via $V_n$ defined in Lemma~\ref{lem:lyapunov-CW-killed}, with $\gamma=a$ and $K=\frac{b}{n}$. For the second assumption, we set $R=\frac{3b}{(1-a)n}>\frac{2K}{1-\gamma}$, and notice that 
\begin{align*}
    h_n^{-1}(m)g(m)\leqslant \frac{3b}{(1-a)n} \quad \implies\quad g(m)\leqslant \frac{3b}{(1-a)n}
    \implies\quad (m-m_+)^2\leqslant \frac{C}{n},
\end{align*}
for some constant $C>0$ depending only on $a$, $b$ and the function $g$. We may thus use Lemma~\ref{lem:densite-CW-killed}, and choose $\omega>0$ large enough so that
\begin{align*}
    V_n(m)\leqslant R\quad \implies\quad  m\in K_n.
\end{align*}
Hence, we satisfy the second assumption of Proposition~\ref{prop:yathhm} with $\alpha=c$ and $\nu$ given in Lemma~\ref{lem:densite-CW-killed}. We therefore obtain the existence of a unique invariant measure $\tilde{\nu}_\infty^n$ on $E_n^{\varepsilon}$ for $P^n_\tau$. We now consider 
\begin{align*}
    \alpha_0=\frac{c}{2},\quad \gamma_0=\frac{3+a}{4},\quad \xi_n=\frac{cn}{2b},\quad \bar{\alpha}=\max\left(1-\frac{c}{2}, \frac{16(1-a)+3c(3+a)}{16(1-a)+12c}\right)<1,
\end{align*}
where we write $\xi_n$ in order to insist on the dependence on $n$ (whereas the other parameters $\alpha_0$, $\gamma_0$ and $\bar{\alpha}$ are independent of $n$). From now on, for the sake of conciseness, write 
\[
d_n:=d^{TV}_{\xi_n,V_n, E^\varepsilon_n}.
\]
We obtain that for any probability measures $\mu$ and $\nu$ on $E_n^{\varepsilon}$
\begin{align*}
    d_n(\mu P^n_\tau ,\nu P^n_\tau )\leqslant \bar{\alpha} d_n( \mu, \nu).
\end{align*}

\textbf{$\bullet$ Contraction for all $t>0$.} Fix now $\eta\in]\varepsilon,m_+[$. 
For any $t\geqslant 0$ and any initial point $m^n_0\in[\eta,1]\cap E^\varepsilon_n$, we have
\begin{align*}
    d_n(\delta_{m^n_0} P^n_t ,\tilde{\nu}^n_\infty )\leqslant& d_n\left(\left(\delta_{m^n_0} P^n_{t-\left \lfloor \frac{t}{\tau}\right\rfloor \tau}\right) P^n_{\left \lfloor \frac{t}{\tau}\right\rfloor \tau}, \delta_{m^n_0} P^n_{\left \lfloor \frac{t}{\tau}\right\rfloor \tau}  \right)+d_n\left(\delta_{m^n_0} P^n_{\left \lfloor \frac{t}{\tau}\right\rfloor \tau} ,\tilde{\nu}^n_\infty \right)\\
    \leqslant& \bar{\alpha}^{\left \lfloor \frac{t}{\tau}\right\rfloor } d_n\left( \delta_{m^n_0} P^n_{t-\left \lfloor \frac{t}{\tau}\right\rfloor \tau}, \delta_{m^n_0}\right)+\bar{\alpha}^{\left \lfloor \frac{t}{\tau}\right\rfloor } d_n\left(\delta_{m^n_0},\tilde{\nu}^n_\infty \right).
\end{align*}
At this stage, one could try and bound 
\begin{align*}
d_n(\delta_{m^n_0} P^n_t ,\tilde{\nu}^n_\infty )\leqslant& \frac{4(1+\xi_n \|V_n\|_{\infty})}{\bar{\alpha}}e^{\frac{\ln(\bar{\alpha})}{\tau} t},
\end{align*}
but the main issue would then be the term $\|V_n\|_{\infty}$, since $V_n=h_n^{-1}g$ and, denoting $\varepsilon_n$ the left-most point in $E^\varepsilon_n$, we might have $h_n(\varepsilon_n)\xrightarrow[n\rightarrow \infty]{} 0$. We therefore need to be more careful. Note that we have
\begin{align*}
d_n\left(\delta_{m^n_0},\tilde{\nu}^n_\infty \right)\leqslant&2+ \xi_n \left(V_n(m^n_0)+\tilde{\nu}^n_\infty(V_n)\right).
\end{align*}
We now bound $V_n(m^n_0)$, $\tilde{\nu}^n_\infty(V_n)$ and $\delta_{m^n_0}P^n_{t-\left \lfloor \frac{t}{\tau}\right\rfloor \tau}(V_n)$ uniformly in $n$ and $t$.
Using Lemma~\ref{lem:estime-hn} and the fact that $m^n_0\geqslant \eta>\varepsilon$ for all $n\in\n$, we have 
\begin{align*}
\limsup_n V_n(m^n_0)\leqslant \sup_{[\eta,1]}g\leqslant \|g\|_\infty.
\end{align*}
In particular, $V_n(m^n_0)$ is bounded uniformly in $n$. Then, we use Lemma~\ref{lem:lyapunov-CW-killed} to get that, for any $k\in\mathbb{N}\setminus\{0\}$
\begin{align*}
P^n_{k\tau}(V_n)(m^n_0)\leqslant  P^n_{(k-1)\tau}\left(aV_n(m^n_0)+\frac{b}{n}\right)\leqslant a^k V_n(m^n_0)+\frac{b}{n}\frac{1-a^k}{1-a}\leqslant V_n(m^n_0) +\frac{b}{(1-a)n}.
\end{align*}
In particular, again since $V_n(m^n_0)$ is bounded uniformly in $n$, we obtain that $P^n_{k\tau}(V_n)(m^n_0)$ is upper-bounded uniformly in $k$ and $n$, say by a constant $C_\infty$. Since, $\|\xi_nV_n\|_{L^\infty(\xi_n, V_n, E^\varepsilon_n)}\leqslant 1$, we obtain
\begin{align*}
\xi_n\left(\tilde{\nu}^n_\infty(V_n)-C_\infty\right)\leqslant \xi_n\left|P^n_{k\tau}(V_n)(m^n_0)-\tilde{\nu}^n_\infty(V_n)\right|\leqslant d_n\left(\delta_{m^n_0}P^n_{k\tau}, \tilde{\nu}^n_\infty\right)\xrightarrow[k\rightarrow\infty]{}0.
\end{align*}
In particular, there exists a constant, also denoted $C_\infty$, such that for all $n\in\n$
\begin{equation}\label{eq:borne_unif_V_infty}
   \tilde{\nu}^n_\infty(V_n) \leqslant C_\infty.
\end{equation}
Note that, again using Lemma~\ref{lem:lyapunov-CW-killed} (noticing that the parameter $b$ is independent of $\tau$), we can also bound $\delta_{m^n_0}P^n_{s}(V_n)$ uniformly in $n$ and $s\in[0,1]$. In particular, we now obtain that there exist some constants $C,c>0$ (depending only on $\varepsilon$, $\eta$, $\tau$ and $g$) such that for all $n\in\n$, all $t\geqslant 0$ and all $m^n_0\in[\eta,1]$
\begin{equation*}
   d_n\left(\delta_{m^n_0}P^n_t,\tilde{\nu}^n_\infty \right)\leqslant Cne^{-ct}.
\end{equation*}
Thus, for any function $f:E^\varepsilon_n\to\R$
\begin{equation*}
\left| P_t^nf(m^n_0) - \tilde{\nu}^n_\infty(f) \right| \leqslant Cne^{-c t} \left\|\frac{f}{1+\xi_nV_n}\right\|_{\infty}.
\end{equation*}

\textbf{$\bullet$ Conclusion.} We now wish to obtain a similar result on the non-conservative semi-group $M^n$. Changing $f$ into $h_n^{-1}f$ yields for all $m\in E^\varepsilon_n\cap[\eta,1]$
\begin{equation*}
\left| e^{b_n t}M_t^n(f)(m) - h_n(m)\tilde{\nu}^n_{\infty}(h_n^{-1}f) \right| \leqslant Cne^{-ct}h_n(m)\left\|\frac{h_n^{-1}f}{1+\xi_nV_n}\right\|_{\infty}. 
\end{equation*}
Define, for all $m\in E^\varepsilon_n$, the probability measure
\begin{equation*}
\nu^n_{\infty}(m)=\frac{h_n^{-1}(m)\tilde{\nu}^n_{\infty}(m)}{\tilde{\nu}^n_{\infty}(h_n^{-1})}.
\end{equation*}
This measure is such that for any function $f:E^\varepsilon_n\to\R$ we have
\begin{align*}
\nu^n_{\infty}(f)=\sum_{m\in E^\varepsilon_n}f(m)\nu^n_{\infty}(m)=\sum_{m\in E^\varepsilon_n}\frac{f(m)h_n^{-1}(m)\tilde{\nu}^n_{\infty}(m)}{\tilde{\nu}^n_{\infty}(h_n^{-1})}=\frac{\tilde{\nu}^n_{\infty}(fh_n^{-1})}{\tilde{\nu}^n_{\infty}(h_n^{-1})},
\end{align*}
which yields in particular $\nu^n_{\infty}(h_n)\tilde{\nu}^n_{\infty}(h_n^{-1})=1$.
This implies that
\begin{equation}\label{eq:long-time-semi-group}
\left| M_t^n(f)(m) - e^{-b_n t}h_n(m)\frac{\nu^n_{\infty}(f)}{\nu^n_{\infty}(h_n)} \right| \leqslant Cne^{-ct}h_n(m)e^{-b_n t}\left\|\frac{h_n^{-1}f}{1+\xi_nV_n}\right\|_{\infty}. 
\end{equation}
Notice that this inequality is more precise than \eqref{eq:conv-non-explicite}. Now, write
\begin{align*}
M^n_t(\mathds{1})(m)=\|h_n\|_\infty M^n_t(\mathds{1})(m)\geqslant M^n_th_n(m)=e^{-b_nt}h_n(m),
\end{align*}
and thus for all bounded $f:[0,1]\to \R$
\begin{align}
\left|\mathbb{E}_m\left(f(m_t^n)\Big|\tau_n>t\right)-\nu^n_\infty(f)\right| =& \left| \frac{M^n_t(f)(m)}{M^n_t(\mathds{1})(m)} -  \nu^n_\infty(f)\right|\nonumber \\ 
=& \frac{1}{M^n_t(\mathds{1})(m)}\left| M^n_t(f)(m) - M^n_t(\mathds{1})(m)\nu^n_\infty(f) \right|\nonumber \\ 
\leqslant& e^{b_n t}h_n^{-1}(m) \left| M^n_t(f)(m) - e^{-b_n t}h_n(m)\frac{\nu^n_\infty(f)}{\nu^n_{\infty}(h_n)} \right|\nonumber \\
&+e^{b_n t}h_n^{-1}(m) \nu^n_\infty(f)\left| h_n(m)e^{-b_n t}\frac{1}{\nu^n_{\infty}(h_n)} - M^n_t(\mathds{1})(m) \right|\nonumber\\ 
 \leqslant& Cne^{-c t}\left\|\frac{h_n^{-1}f}{1+\xi_nh_n^{-1}g}\right\|_{\infty} +Cne^{-c t} \|f\|_{\infty} \left\|\frac{h_n^{-1}}{1+\xi_nh_n^{-1}g}\right\|_{\infty}\label{eq:unif_en_n_pour_f_particulier},
\end{align}
where this last inequality was obtained thanks to \eqref{eq:long-time-semi-group}. Note that, again since $h_n$ converges to $1$ uniformly in all compact sets included in $]\varepsilon,m_+[$ and since $g$ is explicit and only reaches $0$ at $m_+$, there exists a constant $C>0$, independent of $n$, such that
\begin{align*}
\left\|\frac{h_n^{-1}}{1+\xi_nh_n^{-1}g}\right\|_{\infty}=\left\|\frac{1}{h_n+\xi_ng}\right\|_{\infty} \leqslant C.
\end{align*}
We thus obtain that there exists $C, c>0$ such that for all $f$ satisfying $\|f\|_{\infty}\leqslant 1$, and for all $m\in[\eta,1]\cap E^\varepsilon_n$, we have
\begin{align*}
\left|\mathbb{E}_m\left(f(m_t^n)\Big|\tau_n>t\right)-\nu^n_\infty(f)\right| \leqslant Cne^{-ct},
\end{align*}
which is the desired result. 
\end{proof}

%
%
%
%

\section{Propagation of Chaos}\label{sec:PoC}

The goal of this section is to prove our main theorem, Theorem~\ref{thm:final_result}, building upon Theorem~\ref{thm:long-time-behavior}. To this end, we need two intermediate propagation of chaos results. Recall that $\mathcal{L}(m^n_t)$ denotes the law of $m^n_t$, and  $\mathcal{L}(\bar{m}_t)$ the law of $\overline{m}_t$. Then, consider the two following estimates

\begin{enumerate}
\item \textbf{Short time control.} There exist $c_1, C_1>0$ such that for any $n\in\mathbb{N}$, any Lipschitz continuous function $f:[0,1]\to\R$, and any $t\geqslant 0$
\begin{align}\label{eq:non-unif-poc}
&\left|\E_{\mathcal{L}(m^n_0)}\left(f(m_t^n)|\tau_n>t\right)-\E_{\mathcal{L}(\bar{m}_0)} f(\overline{m}_t)\right|\nonumber\\
&\qquad\qquad\leqslant C_1\frac{\|f\|_{\textrm{lip}} }{\mathbb{E}_{\mathcal{L}(m^n_0)}\left(\mathds{1}_{\tau_n>t}\right)}\left( e^{c_1t}\mathcal{W}_2(\mathcal{L}(m^n_0),\mathcal{L}(\bar{m}_0))+\frac{e^{c_1t}}{\sqrt{n}}+1-\mathbb{E}_{\mathcal{L}(m^n_0)}\left(\mathds{1}_{\tau_n>t}\right)\right).
\end{align}
\item \textbf{Intermediate control.} Let $\eta\in]\varepsilon,m_+[$. There exist $c_2, c_3, C_2>0$ such that for any $n\in\mathbb{N}$, for any $m^n_0,\overline{m}_0\in[\eta,1]$, any Lipschitz continuous function $f:[0,1]\to\R$ and any $t\geqslant 0$
\begin{align}
&\left|\E_{m^n_0}\left(f(m_t^n)|\tau_n>t\right)-f(\overline{m}_t)\right|\nonumber\\
&\qquad\qquad\leqslant C_2\|f\|_{\textrm{lip}}e^{c_3\frac{t}{\sqrt{n}}}\left( e^{-c_2t}|m^n_0-\overline{m}_0|+\frac{1}{\sqrt{n}}+e^{-c_2t}+1-e^{-c_3\frac{t}{\sqrt{n}}}\right).\label{eq:control_int}
\end{align}
\end{enumerate}

Let us provide some comments on these bounds. Using Theorem~\ref{thm:long-time-behavior}, we show that the QSD $\nu_{\infty}^n$ is close to $\delta_{m_+}$, as $n$ goes to infinity (see \eqref{eq:les_deux_sont_proches} below). This implies that, uniformly in $t\gtrsim n^{\alpha}$ for some $\alpha>0$, $\E\left(f(m_t^n)|\tau_n>t\right)$ is close to $\E f(\overline{m}_t)$ because both quantities are close to their respective limit. Then, in \eqref{eq:non-unif-poc}, we consider initial distributions that are such that $\mathbb{E}_{\mathcal{L}(m^n_0)}\left(\mathds{1}_{\tau_n>t}\right)\gtrsim e^{-\frac{ct}{\sqrt{n}}}$ (this is direct if $\mathcal{L}(m^n_0)$ is the QSD, see also Lemma~\ref{lem:estime-bn} in the case $\mathcal{L}(m^n_0)$ is a Dirac mass). The bound~\eqref{eq:non-unif-poc}, which is a direct consequence of Proposition~\ref{prop:PoCnonUnif}, can therefore be used in order to prove convergence uniformly in $t\lesssim \ln(n)$. It now remains to use \eqref{eq:control_int}  to deal with times  $\ln(n) \lesssim t\lesssim n^{\alpha}$. This control consists in a modification of the usual propagation of chaos proof, using the convexity of the underlying potential $g$ in an interval around $m_+$ and a Lyapunov type condition that tends to bring the processes in said interval. The bound~\eqref{eq:control_int}  is therefore a result akin to \textit{generation of chaos}: for the error $\left|\E\left(f(m_t^n)|\tau_n>t\right)-f(\overline{m}_t)\right|$ to be small enough, independently of the initial condition, it is sufficient to consider both $n$ and $t$ large enough (with $t$ still smaller than $\sqrt{n}$ in our case).

%
%
%
%

\begin{remark}
All constants given in Theorem~\ref{thm:long-time-behavior} and controls~\ref{eq:non-unif-poc}~and~\ref{eq:control_int}, i.e. $c,C$, $c_1,c_2,c_3$ and $C_1,C_2$, can be explicitly computed and depend only on $\beta$, $\eta$ and $\varepsilon$. For the sake of conciseness, and because we would not learn much from their explicit value, we choose not to write them.
\end{remark}

\subsection{Propagation of chaos for the killed process}\label{sec:preuve_estimee_1}

\begin{lemma}
    Estimate~\eqref{eq:non-unif-poc} holds.
\end{lemma}

\begin{proof}
Let $f:[0,1]\to\R$ be Lipschitz continuous, and $m_0^n, \overline m_0$ be random initial conditions. We have (without writing the dependency of the expectation on the initial condition for simplicity)
\begin{align*}
&\left|\E\left(f(m_t^n)|\tau_n>t\right)-\E f(\overline{m}_t)\right|\\
&\qquad=\frac{1}{\E(\mathds{1}_{\tau_n>t})}\left|\E\left(f(m_t^n)\mathds{1}_{\tau_n>t}\right)- \E f(\overline{m}_t)\E(\mathds{1}_{\tau_n>t}) \right|\\
&\qquad=\frac{1}{\E (\mathds{1}_{\tau_n>t})}\left|\E f(m_t^n)-\mathbb{E} f(\overline{m}_t)-\left(\E \left(f(m_t^n)\mathds{1}_{\tau_n\leqslant t}\right)-\E f(\overline{m}_t)\E \left(\mathds{1}_{\tau_n\leqslant t}\right)\right)\right|\\
&\qquad\leqslant \frac{1}{\E (\mathds{1}_{\tau_n>t})}\left|\E f(m_t^n)-\mathbb{E} f(\overline{m}_t)\right|\\
&\qquad\qquad+ \frac{1}{\E (\mathds{1}_{\tau_n>t})}\left|\E \left(f(m_t^n)\mathds{1}_{\tau_n\leqslant t}\right)-\E f(\overline{m}_t)\E \left(\mathds{1}_{\tau_n\leqslant t}\right)\right|.
\end{align*}
First, notice that we have
\begin{multline*}
\left|\E f(m_t^n)-\mathbb{E} f(\overline{m}_t)\right|
\leqslant \|f\|_{\textrm{lip}} \sup_{\|g\|_{\textrm{lip}}\leqslant 1}\left| \E\left(g(m_t^n)-g(\overline{m}_t)\right)\right| \leqslant \|f\|_{\textrm{lip}}\mathcal{W}_1\left(\mathcal{L}(m_t^n), \mathcal{L}(\overline{m}_t)\right)\\  \leqslant \|f\|_{\textrm{lip}}\mathcal{W}_2\left(\mathcal{L}(m_t^n), \mathcal{L}(\overline{m}_t)\right) 
\leqslant \|f\|_{\textrm{lip}} e^{ct}\left(\mathcal{W}_2\left(\mathcal{L}(m^n_0), \mathcal{L}(\overline{m}_0)\right)+\frac{C}{\sqrt{n}}\right),
\end{multline*}
where we used \eqref{eq:PoC_pas_unif} of Proposition~\ref{prop:PoCnonUnif} for the last inequality. Then, we also have
\begin{align*}
    \left|\E\left(f(m_t^n)\mathds{1}_{\tau_n\leqslant t}\right)-\E f(\overline{m}_t)\E \left(\mathds{1}_{\tau_n\leqslant t}\right)\right|
    &=\left|\E \left((f(m_t^n)-f(0))\mathds{1}_{\tau_n\leqslant t}\right)-\E \left(f(\overline{m}_t)-f(0)\right)\E \left(\mathds{1}_{\tau_n\leqslant t}\right)\right|\\
    & \leqslant \E \left(\left|f(m_t^n)-f(0)\right|\mathds{1}_{\tau_n\leqslant t}\right)+\E \left|f(\overline{m}_t)-f(0)\right|\E \left(\mathds{1}_{\tau_n\leqslant t}\right)\\
    & \leqslant 2\|f\|_{\textrm{lip}}\mathbb{E} (\mathds{1}_{\tau_n\leqslant t}).
\end{align*}
Hence the result.
\end{proof}

%
%
%
%

\subsection{An auxiliary process}\label{sec:aux_proc}

The goal of this section is to establish \eqref{eq:control_int}. To do so, we start by constructing two processes, $\mu^n$ and $\overline{\mu}$, which mostly behave like $m^n$ and $\overline{m}$ and for which we prove not a result of uniform in time propagation of chaos, but of \textit{generation of chaos}: convergence holds for $n$ and $t$ large enough, even if the initial conditions do not converge to one another as $n\rightarrow\infty$. Consider the jump rates for $\mu^n_t$ 
\begin{align*}
    m\mapsto m-\frac{2}{n}\quad \text{with rate}\quad \tilde{\lambda}^n_-(m)=\left\{\begin{array}{ll}\lambda^n_-(m)&\text{if }m\geqslant \varepsilon,\\
    \lambda^n_-(\varepsilon)\frac{m}{\varepsilon}&\text{if }m\in[0,\varepsilon]\text{ and }n\text{ even},\\
    \lambda^n_-(\varepsilon)\frac{m-\frac{1}{n}}{\varepsilon-\frac{1}{n}}&\text{if }m\in[0,\varepsilon]\text{ and }n\text{ odd},
    \end{array}\right.
\end{align*}
and
\begin{align*}
    m\mapsto m+\frac{2}{n}\quad \text{with rate}\quad \tilde{\lambda}^n_+(m)=\left\{\begin{array}{ll}\lambda^n_+(m)&\text{if }m\geqslant \varepsilon,\\ \lambda^n_+(\varepsilon)+\left(\tilde{\lambda}_-^n(m)-\lambda^n_-(\varepsilon)\right)&\text{if }m\in[0,\varepsilon].\end{array}\right.
\end{align*}
These rates are chosen such that, provided $m_0^n=\mu_0^n$, a trivial coupling ensures $m_t^n=\mu_t^n$ while $t<\tau_n$. Furthermore, by construction, $\mu^n_t\in[0,1]$ almost surely, as $\tilde{\lambda}^n_-(0)=0$ (resp. $\tilde{\lambda}^n_-(\frac{1}{n})=0$)  if $n$ is even (resp. odd) and thus $0$ (resp. $\frac{1}{n}$) is the left most point of $E^n$  that can be reached by this process. We define its generator by
\begin{align*}
\mathcal{B}_nf(m)=\tilde{\lambda}^n_+(m)\left(f\left(m+\frac{2}{n}\right)-f(m)\right)+\tilde{\lambda}^n_-(m)\left(f\left(m-\frac{2}{n}\right)-f(m)\right).
\end{align*}
When $n\rightarrow\infty$, we may easily check that this generator converges towards
\begin{align*}
\mathcal{B}f(m)=-U'(m)f'(m),
\end{align*}
with
\begin{equation}\label{eq:def_U}
U'(m)=-\frac{2}{n}\left(\tilde{\lambda}^n_+(m)-\tilde{\lambda}^n_-(m)\right)=\left\{\begin{array}{ll}g'(m)&\text{if }m\geqslant \varepsilon,\\ g'(\varepsilon)&\text{if }m\in[0,\varepsilon].\end{array}\right.
\end{equation}
We may then define $\overline{\mu}_t$ as the solution to the ODE
\begin{align*}
    \frac{d}{dt}\overline{\mu}_t=-U'(\overline{\mu}_t).
\end{align*}
Again, provided $\overline{m}_0=\overline{\mu}_0\geqslant \epsilon$, we have  $\overline{m}_t=\overline{\mu}_t$ for all $t\geqslant 0$. We now define $U$ as the primitive of $U'$ such that $U(m_+)=0$, thus ensuring $U=g$ on $[\varepsilon,1]$. See Figure~\ref{fig:potentiel_U}. Let us start by proving a Lyapunov condition


\begin{figure}
    	\centering
	\includegraphics[width=0.5\linewidth]{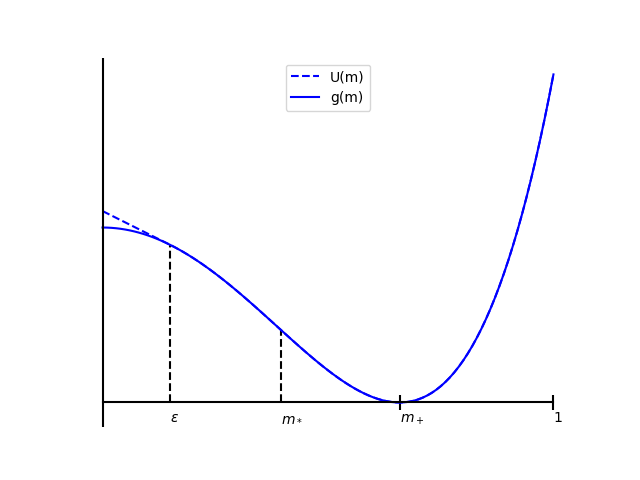}
	\caption{Potential $U$ defined in \eqref{eq:def_U} for $\beta=1.2$. It consists in a modification of $g$ on $[0,\varepsilon]$. Furthermore, we consider $m_*$ a point such that $g''(m)\geqslant g''(m_*)>0$ for all $m\in[m_*,1]$, thus defining a set on which the underlying potential is strictly convex.}
 \label{fig:potentiel_U}
\end{figure}


\begin{lemma}\label{lem:lya_aux}
    There exist $c, C>0$ such that we have for all $t\geqslant 0$ and all $n\in\mathbb{N}$
    \begin{equation}\label{eq:lya_aux}
        \mathbb{E}\left(U(\mu_t^n)\right)\leqslant e^{-ct}\mathbb{E}\left(U(\mu_0^n)\right)+\frac{C}{n},
    \end{equation}
    as well as
    \begin{equation}\label{eq:lya_aux_nl}
        U(\overline{\mu}_t)\leqslant e^{-ct}U(\overline{\mu}_0).
    \end{equation}
\end{lemma}


\begin{proof}
We have $\mathcal{B}U(m)=-(U'(m))^2$ and we can easily check that there exists $c>0$ such that for all $m\in[0,1]$
\begin{align*}
    \frac{\left|U'(m)\right|^2}{U(m)}\geqslant c.
\end{align*}
Therefore $\mathcal{B}U(m)\leqslant -cU(m)$, i.e. $$
\frac{d}{dt}U(\overline{\mu}_t)\leqslant -cU(\overline{\mu}_t),$$
and thus by Gronwall's lemma we obtain \eqref{eq:lya_aux_nl}. Then, we have 
\begin{align*}
\mathcal{B}_nU(m)\leqslant\mathcal{B}U(m)+\frac{\C}{n}\leqslant -cU(m)+\frac{C}{n},
\end{align*}
where $C$ is a constant depending only on $\|U'\|_{\textrm{lip}}$, $\varepsilon$ and $\beta$. Finally, like previously, Gronwall's lemma yields \eqref{eq:lya_aux}.
\end{proof}


\begin{lemma}\label{lem:presque_poc_unif}
There exist $c,C>0$ such that for all $t\geqslant 0$, all $n\geqslant 0$ and all $\mu_0^n,\overline{\mu}_0\in[0,1]$, we have
\begin{align*}
    \mathcal{W}_2(\mathcal{L}(\mu_t^n),\delta_{\overline{\mu}_t})\leqslant e^{-ct}\left(|\mu_0^n-\overline{\mu}_0|+C\right)+\frac{C}{\sqrt{n}}.
\end{align*}
\end{lemma}


\begin{proof}
Define $m_*\in]\varepsilon,m_+[$ such that there exists some $c>0$ such that $U''(m)=g''(m)\geqslant c$ for all $m\geqslant m_*$. Such a point always exists (see Figure~\ref{fig:potentiel_U}) as a simple calculation ensures that $g'''(m)>0$ for $m>0$, with additionally $g''(m_+)>0$ by Lemma~\ref{lem:mini-non-degenere}. Let us compute
\begin{align*}
\frac{d}{dt}\mathbb{E}\left|\mu_t^n-\overline{\mu}_t\right|^2\leqslant&-2\mathbb{E}\left((U'(\mu_t^n)-U'(\overline{\mu}_t)(\mu_t^n-\overline{\mu}_t)\right)+\frac{8e^\beta}{n}\\
=&-2\mathbb{E}\left((U'(\mu_t^n)-U'(\overline{\mu}_t)(\mu_t^n-\overline{\mu}_t)\mathds{1}_{\mu_t^n\geqslant m_*\text{ and }\overline{\mu}_t\geqslant m_*}\right)\\
&-2\mathbb{E}\left((U'(\mu_t^n)-U'(\overline{\mu}_t)(\mu_t^n-\overline{\mu}_t)\mathds{1}_{\mu_t^n\leqslant m_*\text{ or }\overline{\mu}_t\leqslant m_*}\right)\\
&+\frac{8e^\beta}{n}.
\end{align*}
First, since $U''(m)\geqslant c>0$ for all $m\geqslant m_*$, we have
\begin{align*}
&-2\mathbb{E}\left((U'(\mu_t^n)-U'(\overline{\mu}_t)(\mu_t^n-\overline{\mu}_t)\mathds{1}_{\mu_t^n\geqslant m_*\text{ and }\overline{\mu}_t\geqslant m_*}\right)\\
&\qquad\qquad\leqslant -2c \mathbb{E}\left(\left|\mu_t^n-\overline{\mu}_t\right|^2\mathds{1}_{\mu_t^n\geqslant m_*\text{ and }\overline{\mu}_t\geqslant m_*}\right)\\
&\qquad\qquad= -2c \mathbb{E}\left|\mu_t^n-\overline{\mu}_t\right|^2 
+2c \mathbb{E}\left(\left|\mu_t^n-\overline{\mu}_t\right|^2\mathds{1}_{\mu_t^n\leqslant m_*\text{ or }\overline{\mu}_t\leqslant m_*}\right).
\end{align*}
Likewise, since $\|U'\|_{\textrm{lip}} < \infty$, we have 
\begin{align*}
-2\mathbb{E}\left((U'(\mu_t^n)-U'(\overline{\mu}_t)(\mu_t^n-\overline{\mu}_t)\mathds{1}_{\mu_t^n\leqslant m_*\text{ or }\overline{\mu}_t\leqslant m_*}\right)\leqslant 2\|U'\|_{\textrm{lip}} \mathbb{E}\left(\left|\mu_t^n-\overline{\mu}_t\right|^2\mathds{1}_{\mu_t^n\leqslant m_*\text{ or }\overline{\mu}_t\leqslant m_*}\right).
\end{align*}
It now remains to compute
\begin{align*}
\mathbb{E}\left(\left|\mu_t^n-\overline{\mu}_t\right|^2\mathds{1}_{\mu_t^n\leqslant m_*\text{ or }\overline{\mu}_t\leqslant m_*}\right)\leqslant \mathbb{E}\left(\mathds{1}_{\mu_t^n\leqslant m_*}\right)+\mathds{1}_{\overline{\mu}_t\leqslant m_*}.
\end{align*}
Note that $\{m\leqslant m_*\}\implies\{U(m)\geqslant U(m_*)\}$ and thus, thanks to Lemma~\ref{lem:lya_aux}, there exist $\tilde{c}, \tilde{C}>0$ such that
\begin{align*}
    \mathbb{E}\left(\mathds{1}_{\mu_t^n\leqslant m_*}\right)\leqslant \mathbb{E}\left(\frac{U(\mu_t^n)}{U(m_*)}\right)\leqslant e^{-\tilde{c}t}\frac{U(\mu_0^n)}{U(m_*)}+\frac{\tilde{C}}{U(m_*)n}.
\end{align*}
Likewise
\begin{align*}
\mathds{1}_{\overline{\mu}_t\leqslant m_*}\leqslant \frac{U\left(\overline{\mu}_t\right)}{U(m_*)}\leqslant e^{-\tilde{c}t}\frac{U\left(\overline{\mu}_0\right)}{U(m_*)}.
\end{align*}
Finally we get
\begin{align*}
    \frac{d}{dt}\mathbb{E}\left|\mu_t^n-\overline{\mu}_t\right|^2\leqslant& -2c\mathbb{E}\left|\mu_t^n-\overline{\mu}_t\right|^2+\frac{1}{n}\left(8e^{\beta}+2\frac{\tilde{C}(c+\|U'\|_{\textrm{lip}})}{U(m_*)}\right)\\
    &+2e^{-\tilde{c}t}\left(c+\|U'\|_{\textrm{lip}}\right)\left(\frac{U(\mu_0^n)}{U(m_*)}+\frac{U\left(\overline{\mu}_0\right)}{U(m_*)}\right).
\end{align*}
Note that, because $U$ is bounded and $U(m_*)>0$, this actually means that there exists $c,\tilde{c},C>0$ such that
\begin{align*}
\frac{d}{dt}\mathbb{E}\left|\mu_t^n-\overline{\mu}_t\right|^2\leqslant& -2c\mathbb{E}\left|\mu_t^n-\overline{\mu}_t\right|^2+e^{-\tilde{c}t}C+\frac{C}{n},
\end{align*}
and, up to modifying the constant and losing some precision, we can assume that $2c\neq \tilde{c}$. Computing the derivative of the function
\begin{align*}
    t\mapsto e^{2ct}\left(\mathbb{E}\left|\mu_t^n-\overline{\mu}_t\right|^2-\frac{C}{2c-\tilde{c}}e^{-\tilde{c}t}-\frac{C}{2cn}\right)
\end{align*}
we obtain that it is non increasing, and hence the result.
\end{proof}


\begin{proof}[Proof of bound~\eqref{eq:control_int}]
We may now prove \eqref{eq:control_int}. Similarly as in Section~\ref{sec:preuve_estimee_1}, let $f:[0,1]\to\R$ be a Lipschitz test function and compute
\begin{align}
\left|\E_{m^n_0}\left(f(m_t^n)|\tau_n>t\right)-f(\overline{m}_t)\right|=&\frac{1}{\E_{m^n_0}(\mathds{1}_{\tau_n>t})}\left|\E_{m^n_0}\left(\left(f(m_t^n)-f(\overline{m}_t)\right)\mathds{1}_{\tau_n>t}\right)\right|\nonumber\\
=&\frac{1}{\E_{m^n_0}(\mathds{1}_{\tau_n>t})}\left|\E_{m^n_0}\left(\left(f(\mu_t^n)-f(\overline{\mu}_t)\right)\mathds{1}_{\tau_n>t}\right)\right|\nonumber\\
\leqslant& \frac{1}{\E_{m^n_0}(\mathds{1}_{\tau_n>t})}\left|\E_{m^n_0}\left(f(\mu_t^n)-f(\overline{\mu}_t)\right)\right|\nonumber\\
&+ \frac{1}{\E_{m^n_0}(\mathds{1}_{\tau_n>t})}\left|\E_{m^n_0}\left(\left(f(\mu_t^n)-f(\overline{\mu}_t)\right)\mathds{1}_{\tau_n\leqslant t}\right)\right|.\label{eq:gen_chaos_int_0}
\end{align}
Using Lemma~\ref{lem:presque_poc_unif} we get
\begin{equation}
\left|\E_{m^n_0}\left(f(\mu_t^n)-f(\overline{\mu}_t)\right)\right|
\leqslant \|f\|_{\textrm{lip}}\mathcal{W}_2\left(\mathcal{L}(\mu_t^n), \delta_{\overline{\mu}_t}\right)\nonumber
\leqslant \|f\|_{\textrm{lip}} e^{-ct}\left(\left|m_0^n-\overline{m}_0\right|+C\right)+\frac{C}{\sqrt{n}},\label{eq:gen_chaos_int_1}
\end{equation}
Then, we also have 
\begin{align}
\left|\E_{m^n_0}\left(\left(f(\mu_t^n)-f(\overline{\mu}_t)\right)\mathds{1}_{\tau_n\leqslant t}\right)\right|\leqslant \E_{m^n_0}\left(\left|f(\mu_t^n)-f(\overline{\mu}_t)\right|\mathds{1}_{\tau_n\leqslant t}\right)
\leqslant \|f\|_{\textrm{lip}}\mathbb{P}_{m^n_0}(\tau_n\leqslant t).\label{eq:gen_chaos_int_2}
\end{align}
It now remains to control $\mathbb{P}_{m^n_0}\left(\tau_n>t\right)=\E_{m^n_0}(\mathds{1}_{\tau_n>t})$.
By definition of $h_n$
\begin{align*}
\E_{m^n_0}(\mathds{1}_{\tau_n>t})=M^n_t\mathds{1}(m^n_0)\geqslant M^n_th_n(m^n_0) = e^{-b_nt}h_n(m^n_0).
\end{align*}
By Lemma~\ref{lem:estime-hn}, since $m^n_0\in[\eta,1]$, there exists $C$ independent of $n$ and $m^n_0$ (but not of $\eta$) such that $h_n(m^n_0)\geqslant C$. Then, by Lemma~\ref{lem:estime-bn}, there exist $c>0$ such that $0<b_n\leqslant\frac{c}{\sqrt{n}}$. Therefore, 
\begin{equation}\label{eq:gen_chaos_int_3}
\E_{m^n_0}(\mathds{1}_{\tau_n>t})=\mathbb{P}_{m^n_0}\left(\tau_n>t\right)\geqslant Ce^{-c\frac{t}{\sqrt{n}}}.
\end{equation}
Plugging~\eqref{eq:gen_chaos_int_1},~\eqref{eq:gen_chaos_int_2}~and~\eqref{eq:gen_chaos_int_3} back into~\eqref{eq:gen_chaos_int_0} yields~\eqref{eq:control_int} and hence concludes the proof.
\end{proof}

%
%
%
%

\subsection{Proof of Theorem~\ref{thm:final_result}}

We now prove the main result, Theorem~\ref{thm:final_result}.

\begin{proof}[Proof of Theorem~\ref{thm:final_result}] 

\textbf{$\bullet$ Long-time convergence of $\overline{m}_t$.} Let us start by showing that there exist $C,c>0$ such that for all $t\geqslant 0$
\begin{equation}\label{eq:conv-limit-depart-qsd}
\left|\mathbb{E}_{\nu^n_\infty}(f(\overline{m}_t))-f(m_+)\right|\leqslant C \|f\|_{\textrm{lip}} e^{-ct}.
\end{equation}
In fact, for any initial condition $\overline{m}_0\in[\varepsilon,1]$, we have
\begin{align*}
\left|f(\overline{m}_t)-f(m_+)\right|\leqslant C  \|f\|_{\textrm{lip}}e^{-ct}.
\end{align*}
The proof is very similar to that of Lemmas~\ref{lem:lya_aux}~and~\ref{lem:presque_poc_unif}.
Let us prove the second point, as integrating over $\nu^n_\infty$ would then yield the result. In fact, since
\begin{align*}
    \left|f(\overline{m}_t)-f(m_+)\right|\leqslant \|f\|_{\textrm{lip}}\left|\overline{m}_t-m_+\right|,
\end{align*}
we only have to control the difference between $\overline{m}_t$ and $m_+$. Define $m_*\in]\varepsilon,m_+[$ such that there exists some $c>0$ such that $g''(m)\geqslant c$ for all $m\geqslant m_*$. We have
\begin{align*}
    \frac{d}{dt}\left|\overline{m}_t-m_+\right|^2=&-2g'(\overline{m}_t)(\overline{m}_t-m_+)\\
    =&-2\left(g'(\overline{m}_t)-g'(m_+)\right)\left(\overline{m}_t-m_+\right)\mathds{1}_{\overline{m}_t\geqslant m_*}-2g'(\overline{m}_t)\left(\overline{m}_t-m_+\right)\mathds{1}_{\overline{m}_t< m_*}\\
    \leqslant& -2c |\overline{m}_t-m_+|^2\mathds{1}_{\overline{m}_t\geqslant m_*}+2\|g'\|_\infty\mathds{1}_{\overline{m}_t< m_*}\\
    \leqslant& -2c |\overline{m}_t-m_+|^2+2(c+\|g'\|_\infty)\mathds{1}_{\overline{m}_t< m_*}.
\end{align*}
We then have for some $c'>0$, similarly as Lemma~\ref{lem:lya_aux}
\begin{align*}
    \mathds{1}_{\overline{m}_t< m_*}\leqslant \frac{g(\overline{m}_t)}{g(m_*)}\leqslant e^{-c't}\frac{g(\overline{m}_0)}{g(m_*)}.
\end{align*}
We may now conclude like we concluded Lemma~\ref{lem:presque_poc_unif}, that is by using the fact that the function
\begin{align*}
t\longrightarrow e^{2ct}\left(\left|\overline{m}_t-m_+\right|^2-\frac{2(c+\|g'\|_\infty)g(\overline{m}_0)}{g(m_*)(2c-c')}e^{-c't}\right),
\end{align*}
is non-increasing, which concludes the proof of~\eqref{eq:conv-limit-depart-qsd}.

\textbf{$\bullet$ Convergence of $\nu^n_\infty$ to $\delta_{m_+}$.} Consider a Lipschitz continuous function $f:[0,1]\mapsto \mathbb{R}$. Because $[0,1]$ is compact, we know that $f$ bounded. Let us prove that the quasi-stationary distribution $\nu^n_\infty$ is close to $\delta_{m_+}$. Consider $m_0^n$ distributed according to $\nu^n_\infty$. We have for all $t\geqslant 0$
\begin{align}
\left|\nu^n_\infty(f)-f(m_+)\right|=&\left|\mathbb{E}_{\nu^n_\infty}\left(f(m_0^n)\Big|\tau_n>0\right)-f(m_+)\right|\nonumber\\
=&\left|\mathbb{E}_{\nu^n_\infty}\left(f(m_t^n)\Big|\tau_n>t\right)-f(m_+)\right|\text{ since }\nu^n_\infty\text{ is (quasi-)stationary}.\nonumber\\
\leqslant&\left|\mathbb{E}_{\nu^n_\infty}\left(f(m_t^n)\Big|\tau_n>t\right)-\mathbb{E}_{\nu^n_\infty}(f(\overline{m}_t))\right|\nonumber\\
&+\left|\mathbb{E}_{\nu^n_\infty}(f(\overline{m}_t))-f(m_+)\right|,\label{eq:preuve_final_int_1}
\end{align}
where $\overline{m}_t$ is considered with initial condition $m_0^n$. First, by estimate \eqref{eq:non-unif-poc}
\begin{equation}\label{eq:preuve_final_int_2}
\left|\mathbb{E}_{\nu^n_\infty}\left(f(m_t^n)\Big|\tau_n>t\right)-\mathbb{E}_{\nu^n_\infty}(f(\overline{m}_t))\right|\leqslant C_1\frac{\|f\|_{\textrm{lip}} }{\mathbb{E}_{\nu^n_\infty}\left(\mathds{1}_{\tau_n>t}\right)}\left(\frac{e^{c_1t}}{\sqrt{n}}+1-\mathbb{E}_{\nu^n_\infty}\left(\mathds{1}_{\tau_n>t}\right)\right).
\end{equation}
Because $\nu^n_{\infty}$ is a QSD, we have that $\mathbb{E}_{\nu^n_\infty}\left(\mathds{1}_{\tau_n>t}\right)=e^{-b_nt}$. See for instance \cite[Proposition~2]{MV12} for a proof of this fact. Hence, there exists $c_2>0$ such that $\mathbb{E}_{\nu^n_\infty}\left(\mathds{1}_{\tau_n>t}\right)\geqslant e^{-c_2\frac{t}{\sqrt{n}}}$. Plugging \eqref{eq:conv-limit-depart-qsd} and \eqref{eq:preuve_final_int_2} back into \eqref{eq:preuve_final_int_1}, we obtain the existence of $c_1, c_2, c_3, C>0$ such that for all $n$ and all $t\geqslant 0$
\begin{align*}
\left|\nu^n_\infty(f)-f(m_+)\right|\leqslant C \|f\|_{\textrm{lip}} \left(e^{c_2\frac{t}{\sqrt{n}}}\left(1-e^{-c_2\frac{t}{\sqrt{n}}}\right)+\frac{e^{c_2\frac{t}{\sqrt{n}}+c_1t}}{\sqrt{n}}+e^{-c_3t}\right).
\end{align*}
Choosing $t=\frac{1}{2(c_1+c_2+c_3)}\ln(n)$ we get
\begin{align*}
    \frac{e^{c_2\frac{t}{\sqrt{n}}+c_1t}}{\sqrt{n}}+e^{-c_3t}\leqslant \frac{e^{(c_1+c_2)t}}{\sqrt{n}}+e^{-c_3t}\leqslant n^{-\frac{c_3}{2(c_1+c_2+c_3)}}+n^{-\frac{1}{2}+\frac{c_1+c_2}{2(c_1+c_2+c_3)}}=2n^{-\frac{c_3}{2(c_1+c_2+c_3)}},
\end{align*}
and
\begin{align*}
e^{c_2\frac{t}{\sqrt{n}}}\left(1-e^{-c_2\frac{t}{\sqrt{n}}}\right)\leqslant \frac{c_2te^{c_2\frac{t}{\sqrt{n}}}}{\sqrt{n}}\leqslant \frac{c_2 e^{\frac{c_2}{2(c_1+c_2+c_3)}}}{2(c_1+c_2+c_3)}\frac{\ln(n)}{\sqrt{n}}.
\end{align*}
Therefore, there exist $\alpha, C>0$ such that
\begin{equation}\label{eq:les_deux_sont_proches}
    |\nu^n_\infty(f)-f(m_+)|\leqslant \frac{C\|f\|_{\textrm{lip}}}{n^\alpha}.
\end{equation}

\textbf{$\bullet$ Conclusion.} Fix $\eta\in]\varepsilon,m_+[$, $\overline{m}_0=m_0^n\in[\eta,1]$, and set $t_1=\frac{1}{4c_1}\ln(n)$ and $t_2=\frac{n^{1/4}}{c_3}$. For $t\leqslant t_1$, \eqref{eq:non-unif-poc} yields, using \eqref{eq:gen_chaos_int_3}, that there exist $c_1,c_2, C_1>0$ such that
\[
\left|\E_{m^n_0}\left(f(m_t^n)|\tau_n>t\right)-f(\overline{m}_t)\right|\leqslant C_1\|f\|_{\textrm{lip}} e^{c_2\frac{t}{\sqrt{n}}}\left( \frac{e^{c_1t}}{\sqrt{n}}+1-e^{-c_2\frac{t}{\sqrt{n}}}\right) \leqslant \tilde C_1\frac{\|f\|_{\textrm{lip}}}{n^{\alpha_1}},
\]
for some $\tilde C_1,\alpha_1>0$. For $t_1\leqslant t \leqslant t_2$, \eqref{eq:control_int} and \eqref{eq:gen_chaos_int_3} imply  that there exist $c_3,c_4, C_2>0$ such that
\[
\left|\E_{m^n_0}\left(f(m_t^n)|\tau_n>t\right)-f(\overline{m}_t)\right|\leqslant C_2\|f\|_{\textrm{lip}}e^{c_3\frac{t}{\sqrt{n}}}\left(\frac{1}{\sqrt{n}}+e^{-c_4t}+1-e^{-c_3\frac{t}{\sqrt{n}}}\right)\leqslant \tilde C_2\frac{\|f\|_{\textrm{lip}}}{n^{\alpha_2}},
\]
for some $\tilde C_2,\alpha_1>0$. Finally, for $t\geqslant t_2$, Theorem~\ref{thm:long-time-behavior} and  \eqref{eq:les_deux_sont_proches} yield that there exist $c_5, c_6, C_3, C_4, C_5>0$ such that
\begin{align*}
    &\left|\E_{m^n_0}\left(f(m_t^n)|\tau_n>t\right)-f(\overline{m}_t)\right|\\
    &\hspace{2cm}\leqslant \left|\E_{m^n_0}\left(f(m_t^n)|\tau_n>t\right)-\nu^n_\infty(f)\right|
    +\left|\nu^n_\infty(f)-f(m_+)\right| +\left|f(m_+)-f(\overline{m}_t)\right|\\
    &\hspace{2cm}\leqslant C_3 \|f\|_{\infty} ne^{-c_5t}+\frac{C_4\|f\|_{\textrm{lip}}}{n^\alpha}+ C_5 \|f\|_{\textrm{lip}}e^{-c_6t} \\
    &\hspace{2cm}\leqslant\tilde{C}_3 \frac{\|f\|_{\infty} + \|f\|_{\textrm{lip}}}{n^{\alpha_3}},
\end{align*}
for some $\tilde{C}_3,\alpha_3>0$, which concludes the proof.

\end{proof}

Note that, in the previous proof, a better control of each individual constant and a more precise choice of $t_1$ and $t_2$ could be done in order to optimize the convergence rate in $n$ (optimize, that is, within the limits of the method we use: we do not claim, nor do we believe, that we could achieve the sharpest possible rate of convergence). We choose, for the sake of simplicity and conciseness, to not carry out this program.

\paragraph{Acknowledgements.}
L.J.  is supported by the grant n°200029-21991311 from the Swiss National Science Foundation. P.L.B. is a postdoc at IHES under the  Huawei Young Talents Program.
The two authors would like to sincerely thank Louis-Pierre Chaintron for the many enlightening discussions around this topic, particularly for the proof of Lemma~\ref{lem:estime-hn}.

\bibliographystyle{plain}
\bibliography{biblio}
\end{document}